\documentclass[reqno]{amsart}
\author{Zhenghui Huo}
\usepackage[all]{xy}
\usepackage{enumitem}
\usepackage{amssymb,amsmath,url}
\usepackage{setspace}
\usepackage{amsthm}
\usepackage{cite}
\usepackage{color}
\usepackage[normalem]{ulem}
\begin{document}
\title{The Bergman Kernel on some Hartogs Domains}
\address{Dept. of Mathematics, Univ. of Illinois, 1409 W. Green St., Urbana IL 61801}
\email{huo3@illinois.edu}

\newtheorem{cl}{Claim}[section]
\newtheorem{lem}{Lemma}[section]
\newtheorem{thm}{Theorem}
\newtheorem{ex}{Example}[section]
\newtheorem{de}{Definition}[section]
\newtheorem{co}{Corollary}[section]
\newtheorem*{re*}{Remark}
\begin{abstract}
We obtain new explicit formulas for the Bergman kernel function on two families of Hartogs domains. To do so, we first compute the Bergman kernels on the slices of these Hartogs domains
with some coordinates fixed, evaluate these kernel functions at certain points off the diagonal, and then apply a first order differential operator to them. We find, for example, explicit formulas for the kernel function on $$\{(z_1,z_2,w)\in\mathbb C^3:e^{|w|^2}|z_1|^2+|z_2|^2<1\}$$ and on $$\{(z_1,z_2,w)\in\mathbb C^3:|z_1|^2+|z_2|^2+|w|^2<1+|z_2w|^2\;{\rm and} \;|w|<1\}.$$
We use our formulas to determine the boundary behavior of the kernel function of these domains on the diagonal.

\:\:
\noindent
{\bf AMS Classification Number:} 32A05, 32A07, 32A25, 32A36, 32A40.

\:\:
\end{abstract}
\maketitle
\section{Introduction}
The Bergman kernel, introduced by Stefan Bergman\cite{Bergman}, is a useful tool in the study of several complex variables. A formula for the Bergman kernel in terms of elementary functions is known in only a few cases. If $\{\phi_j\}$ is a complete orthonormal system of $A^2(\Omega)$, then the kernel function on $\Omega$ satisfies
$$K_{\Omega}(z;\bar\zeta)=\sum_{j}\phi_j(z)\overline {\phi_j(\zeta)}.$$
Suppose $\Omega\subseteq\mathbb C^n$, not necessarily bounded, is a Reinhardt domain containing the origin.  For $\alpha=(\alpha_1,\dots,\alpha_n)\in \mathbb N^n$ and $z\in \mathbb C^n$, let $z^{\alpha}$ be the standard multi-index notation. Set $\mathfrak I=\{\alpha:z^{\alpha}\in L^2(\Omega)\}$.  Then $\{z^{\alpha}\}_{\alpha\in\mathfrak I}$ is a complete orthogonal system of $A^2(\Omega)$ and the Bergman kernel $K_{\Omega}$ satisfies:
\begin{equation}\label{1}
K_{\Omega}(z;\bar\zeta)=\sum_{\alpha\in \mathfrak I}\frac{(z\bar\zeta)^{\alpha}}{\|z^{\alpha}\|^2_{L^2(\Omega)}}.
\end{equation}
  In some cases, (\ref{1}) can be expressed in terms of elementary functions. D'Angelo \cite{1978D'A1,1994D'A2} gave the explicit formula of the Bergman kernel function on the domain $\Omega=\{(z,w)\in \mathbb C^{n+m}: \|z\|^2+\|w\|^{2p}<1\}$ for any positive real $p$.
Fransics and Hanges \cite{1996FH1,1997FH2} expressed the Bergman kernel on complex ovals in terms of generalized hypergeometric functions. Here a complex oval is a domain given by $$\{z\in \mathbb C^n:\sum_{j=1}^{n}|z_j|^{2a_j}<1\}$$ where $a_j$'s are positive integers. Park\cite{2008Park1}, by applying the method of \cite{1997FH2}, computed the Bergman kernel on $\{(z,w)\in \mathbb C^2: |z|^4+|w|^4<1\}$ explicitly and proved that the Bergman kernel on $$\{(z,w)\in \mathbb C^2: |z|^{2p_1}+|w|^{2p_2}<1\}$$ for positive integers $p_i$'s has an explicit formulas in terms of elementary functions in only two cases:
  \begin{itemize}
  	\item $p_i=1$ for $i=1$ or $2$.
  	\item $p_1=p_2=2$. 
  \end{itemize}
  In \cite{2013Park2}, Park obtained an explicit formula for the Bergman kernel on the domain $\{(z_1,z_2,z_3)\in \mathbb C^3: |z_1|^{4}+|z_2|^{4}+|z_3|^{4}<1\}$.

Boas, Fu, and Straube \cite{1999BFS} introduced a different method. They considered the domain $\Omega=\{(z,w)\in \mathbb C\times\mathbb C^n:|z|<p(w)\}$ where $p(w)$ is a bounded, positive, continuous function on the interior
of
some bounded domain in
$\mathbb C^n$. By differentiating the Bergman kernel on $\Omega$, they obtained the kernel function on \begin{equation}\label{e}
\{(z,w)\in \mathbb C^m\times\mathbb C^n:\|z\|<p(w)\}.\end{equation}
Additional results have been obtained in \cite{2015Beberok} on the domains
$$\{(z_1,z_2,z_3)\in \mathbb C^{3}:(|z_1|^{2p}+|z_2|^4)^{1/\lambda}+|z_3|^{2/q}<1\},$$
and in \cite{2013Ya} on the Fock-Bargmann-Hartogs domain
$$\{(z,w)\in \mathbb C^{n+m}:\|z\|<e^{-a\|w\|^2}\}.$$

The method used in our paper is new. Using it, we rediscover some of the formulas mentioned above, and we also obtain some new explicit formulas. See Examples 4.2 and 4.3. The operator we use differs from that in \cite{1999BFS}. If we start with the domain $\{(z,w)\in \mathbb C\times\mathbb C^n:|z|<p(w)\}$ from \cite{1999BFS}, the kernel function obtained through our method is on the domain 
$$\{(z_1,w,z_2) \in \mathbb C\times\mathbb C^n\times\mathbb C: |z_1|<p(w)(1-|z_2|^2)^{\alpha}\;{\rm and}\;|z_2|<1\}$$
with $\alpha>0$. By contrast, the result in \cite{1999BFS} applies when $\Omega$ is defined as in (\ref{e}).

We illustrate our idea using the following special case of Example 4.1:

\paragraph{\textbf{Example}}Let $\Omega=\{(z,w)\in \mathbb C^2:|z|^{2a}+|w|^2<1\}$. Regarding $w$ for $|w|<1$ as a parameter, we obtain a family of domains $\{\Omega_{w}\}$ in $\mathbb C$ with $$\Omega_{w}=\Big\{z\in \mathbb C:\frac{|z|^2}{(1-|w|^2)^{\frac{1}{a}}}<1\Big\}.$$
For each $\eta\in \mathbb C$ with $|\eta|<1$, $\Omega_\eta$ is biholomorphic to the unit disk.  Applying  the biholomorphic transformation rule to the Bergman kernel $K_{\Omega_\eta}$ on $\Omega_\eta$ yields:
\begin{equation}\label{b}
K_{\Omega_\eta}(z;\bar \zeta)=\frac{(1-|\eta|^2)^{\frac{1}{a}}}{\pi\big((1-|\eta|^2)^{\frac{1}{a}}-{z\bar{\zeta}}\big)^2}.
\end{equation}
Replacing $z$ in (\ref{b}) by $z(\frac{(1-|\eta|^2)}{(1-w\bar{\eta})})^{\frac{1}{a}}$ and multiplying the right hand side of (\ref{b}) by $(1-|\eta|^2)^{\frac{1}{a}}$ yield a Hermitian symmetric function $K_1$ on $\Omega\times\Omega$:
\begin{equation}\label{c}
K_1(z,w;\bar \zeta,\bar \eta)=\frac{(1-w\bar{\eta})^{\frac{2}{a}}}{\pi\big((1-w\bar{\eta})^{\frac{1}{a}}-{z\bar{\zeta}}\big)^2}.\end{equation}
Let $I$ denote the identity operator. Applying the first order differential operator  $$D_\Omega=\frac{1}{\pi(1-w\bar{\eta})^{2+\frac{1}{a}}}\Big((1+\frac{1}{a})I+\frac{1}{a}z\frac{\partial}{\partial z}\Big),$$
to $K_1$, we obtain $$\frac{(1+a)(1-w\bar{\eta})^{\frac{1}{a}}+(1-a)z\bar{\zeta}}{\pi^2a(1-w\bar{\eta})^{2-\frac{1}{a}}\big((1-w\bar{\eta})^{\frac{1}{a}}-{z\bar{\zeta}}\big)^3}.$$ We then can verify that this function is the Bergman kernel on $\Omega$ (It agrees with the results in \cite{1978D'A1}).

We generalize this technique as follows. We consider certain Hartogs domains $\mathcal U_w$ in $\mathbb C^n$, depending on a parameter $w$ in $\mathbb C$, with the kernel function known on the domain $\Omega=\mathcal U_0$. We call $\Omega$ the ``base" domain. By regarding the parameter $w$ of $\mathcal U_w$ as a new complex variable, we construct a domain $\mathcal U$ in $\mathbb C^n\times\mathbb C$. We call $\mathcal U$ the ``target" domain. As in the example, we obtain the Bergman kernel on $\mathcal U$ by the following procedure:
\begin{enumerate}
\item compute the Bergman kernel $K_{\mathcal U_w}$.
\item evaluate $K_{\mathcal U_w}$ off the diagonal.
\item obtain a Hermitian symmetric function $K$ on $\mathcal U$ by multiplying the result in Step (2) by a certain function.
\item apply a first order differential operator $D_{\mathcal U}$ to $K$.
\item verify that the result in Step (4) is the Bergman kernel on $\mathcal U$.
\end{enumerate}
Our technique works for two kinds of $\mathcal U$ with certain nice properties. We introduce the term ``$n$-star-shaped" for these properties in Section 2, define these two kinds of $\mathcal U$ at the beginning of Section 3, and then demonstrate our formulas for $K_{\mathcal U}$.

In Sections 5 and 6, we investigate the boundary behavior of the kernel functions on the diagonal in some suitable approach regions using the formulas we have obtained. Section 5 uses our explicit formulas in Examples 4.2 and 4.3. Theorems 3 and 4 in Section 6 provide more general results for the case when the slices of the domain are strongly pseudoconvex. In Example 7.1, we apply Theorems 1 and 2 repeatedly, obtaining explicit formulas for the kernel function on rather elaborate domains. We also offer higher dimensional refinements of Theorems 1 and 2.

\paragraph{\bf Acknowledgements}These results are part of the author's PhD thesis at the University of Illinois at Urbana-Champaign. The author acknowledges his thesis advisor Professor John D'Angelo for his patience, encouragement, and valuable advice. The author acknowledges Professor Jeff McNeal for discussions about the boundary behavior of the Bergman kernel. The author also thanks Luke Edholm for helpful conversations. The author thanks the two referees who provided constructive criticisms. This paper is supported by NSF grant DMS 13-61001 of D'Angelo.

\section{Preliminaries}
Let $\Omega$ be a domain in complex Euclidean space $\mathbb C^n$. The space $A^2(\Omega)$, consisting of square integrable holomorphic functions on $\Omega$, is closed in $L^2(\Omega)$ and hence a Hilbert space. We denote by $P$ the orthogonal projection from $L^2(\Omega)$ to $A^2(\Omega)$. Let $\mathbf a=(a_1,\dots,a_n)$ and $z^{\mathbf a}$ be the standard multi-index notation.

We recall the definition of the Bergman kernel. See \cite{Krantz} for more details. Let $z\in \Omega$. For all $f\in A^2(\Omega)$, the map $\delta_z$ from $A^2(\Omega)$ to $\mathbb C$ defined by
\begin{equation}
\delta_z(f)=f(z)
\end{equation}
is a bounded linear functional.
By Riesz's representation theorem, there exists a unique $K_z\in A^2(\Omega)$ such that
\begin{equation}
\delta_z(f)=f(z)=\int_{\Omega}\overline {K_z(\zeta)}f(\zeta)dV(\zeta).
\end{equation}
The Bergman kernel function $K_{\Omega}$ is defined by $K_{\Omega}(z,\bar\zeta)=K_z(\zeta)$. Then 
$$Pf(z)=\int_{\Omega}K_{\Omega}(z,\bar \zeta)f(\zeta)dV(\zeta),\;\;\;\;  \forall f\in L^2(\Omega).$$
These considerations lead to the following lemma (See, e.g., Prop.
1.4.6 in \cite{Krantz}):
\begin{lem}
A function $K:\Omega\times\Omega\rightarrow \mathbb C$ is the Bergman kernel function on $\Omega$ if and only if:
\begin{enumerate}
\item For each $\zeta\in \Omega$, the map $z\mapsto K(z,\bar{\zeta})$ is in $A^2(\Omega)$.
\item $\overline {K(z,\bar w)}=K(w,\bar z)$.\; (Hermitian symmetry)
\item $\int_{\Omega}K(z,\bar w)f(w)dV(w)=f(z)$ for all $f\in A^2(\Omega)$.\; (Reproducing property)
\end{enumerate}
\end{lem}

We will also need the transformation rule for the Bergman kernel under biholomorphic maps. Let $F:\Omega_1\rightarrow\Omega_2$ be biholomorphic. Lemma 2.1 implies the transformation formula: 
\begin{equation}\label{4}
K_{\Omega_1}(z,\bar{w})=JF(z)\overline {JF( w)}K_{\Omega_2}(F(z),\overline {F(w)}),
\end{equation}
where $JF$ is the holomorphic Jacobian determinant of $F$.

If $\{\phi_j\}$ is a complete orthonormal system in $A^2(\Omega)$, then
\begin{equation}
K(z,\bar w)=\sum_{j=1}^{\infty}\phi_j(z)\overline{\phi_j(w)}
\end{equation}
with normal convergence in $\Omega\times \Omega$. 

Let $z_1,\dots,z_n$ and $\zeta_1,\dots,\zeta_m$ denote the first $n$ and last $m$ coordinates in $\mathbb C^{n+m}$. Let $\Omega\subseteq \mathbb C^{n+m}$ be the ``base" domain. Our method of obtaining Bergman kernels requires the space $A^2(\Omega)$ to have a complete orthogonal system of form $\{z^{\mathbf a}\phi_{\mathbf a}(z^{\prime})\}$. This consideration leads us to a class of domains with a symmetry property in the $z$ coordinates. We call these domains $n$-star-shaped Hartogs domains. Before defining them, we recall the definition of Hartogs domain.
\begin{de}
An open connected set $\Omega\subseteq \mathbb C^{n}$ is called a Hartogs domain with symmetric plane $\{z_j=a_j\}$ if $(z_1,\cdots,z_n)\in \Omega$ implies $$(z_1,\dots,z_{j-1},a_j+e^{i\theta}(z_j-a_j),z_{j+1},\dots,z_n)\in \Omega$$ for all $\theta\in\mathbb R$. Such a Hartogs domain $\Omega$ is called complete if
$$(z_1,\dots,z_{j-1},a_j+\lambda(z_j-a_j),z_{j+1},\dots,z_n)\in \Omega$$
for all $\lambda\in \mathbb C$ with $|\lambda|\leq 1$.
\end{de}
If $\Omega$ is a complete Hartogs domain with symmetric plane $\{z_j=0\}$, then the slices of $\Omega$ with all coordinates except $z_j$ fixed are disks in $\mathbb C$. The $n$-star-shaped Hartogs domain is the higher dimensional analogue of such a complete Hartogs domain:
\begin{de}
A domain {$\Omega\subseteq\mathbb C^{n+m}$} is called $n$-star-shaped Hartogs in $(z_1,\dots,z_n)$ if $(z_1,\dots,z_n,\zeta)\in \Omega$ implies 
$\{(\lambda_1z_1,\dots,\lambda_n z_n,\zeta):
|\lambda_j|\leq 1 \:\:{\rm{{for}}}\:\: 1\leq j\leq n\}\subseteq\Omega$.
\end{de}
Let $\Omega\subseteq \mathbb C^{n+m}$ be an $n$-star-shaped Hartogs domain in the first $n$ coordinates. The slices of $\Omega$ with the last $m$ variables fixed are polydisks in $\mathbb C^n$. A holomorphic function $f$ on such a $\Omega$ has the expansion:
\begin{equation*}
f(z,\zeta)=\sum_{\mathbf a}\phi_{\mathbf a}(\zeta)z^{\mathbf a}.
\end{equation*}
Here $\phi_{\mathbf a}(\zeta)$ is holomorphic in $\zeta$ and the series converges normally in $\Omega$.  

Let $\pi:(z,\zeta)\mapsto\zeta$ denote the projection from $\mathbb C^{n}\times \mathbb C^m$ to $\mathbb C^{m}$. For $\zeta\in \pi(\Omega)$, we set $\Omega_{\zeta}=\{z\in \mathbb C^n:(z,\zeta)\in\Omega\}$. The following lemma is a version of Ligocka's result in \cite{Ligocka}. The idea has also appeared in \cite{1999BFS,FR}. In our version, no boundedness condition is assumed for $\Omega$. For convenience, we provide a complete proof below.
\begin{lem}Let $\Omega\subseteq\mathbb C^{n+m}$ be $n$-star-shaped Hartogs in $(z_1,\dots,z_n)$. Then
\begin{enumerate}
\item For any $f\in A^2(\Omega)$, 
\begin{equation*}
f(z,\zeta)=\sum_{\mathbf a}\phi_{\mathbf a}(\zeta)z^{\mathbf a},
\end{equation*}
where for each multi-index $\mathbf a$, $\phi_{\mathbf a}$ is a square-integrable holomorphic function with respect to the measure 
\begin{equation*}
\|z^{\mathbf a}\|^2_{\Omega_{\zeta}}dV(\zeta).
\end{equation*}
\item   If $\{\phi_{\mathbf a,\mathbf b}\}$ is a complete orthogonal system for $A^2(\pi(\Omega),\|z^{\mathbf a}\|^2_{\Omega_{\zeta}})$, then $\{\phi_{\mathbf a,\mathbf b}z^{\mathbf a}\}$ forms a complete orthogonal system for $A^2(\Omega)$.
\end{enumerate}
\end{lem}
\begin{proof}
Let $\{\Omega^k\}$ denote a sequence of compact $n$-star-shaped Hartogs domains such that for all $k$,
the set $\Omega^k\subset\subset\Omega^{k+1}$ and $\bigcup^{\infty}\Omega^k=\Omega$.
Since $f(z,\zeta)=\sum_{\mathbf a}\phi_{\mathbf a}(\zeta)z^{\mathbf a}$ converges normally on $\Omega$, the series $\sum_{\mathbf a}\phi_{\mathbf a}(\zeta)z^{\mathbf a}$ is uniformly convergent on $\Omega^k$. Using polar coordinates, we see that $\phi_{\mathbf a}(\zeta)z^{\mathbf a}\perp\phi_{\mathbf b}(\zeta)z^{\mathbf b}$ in $A^2(\Omega^k)$ if $\mathbf a\neq\mathbf b$.

Thus, for all $k$ and $\phi_{\mathbf a}(\zeta)z^{\mathbf a}\in A^2(\Omega^k)$.
$$\|f(z,\zeta)\|^2_{A^2(\Omega)}\geq\sum_{\mathbf a}\|\phi_{\mathbf a}(\zeta)z^{\mathbf a}\|^2_{A^2(\Omega^k)}.$$
Therefore $\phi_{\mathbf a}(\zeta)z^{\mathbf a}$ is also square integrable on $\Omega$.  Since
\begin{align}
\|\phi_{\mathbf a}(\zeta)z^{\mathbf a}\|^2_{A^2(\Omega)}&=\int_{\pi(\Omega)}|\phi_{\mathbf a}(\zeta)|^2\int_{\Omega_{\zeta}}|z^{\mathbf a}|^2dV(z)\;dV(\zeta)\nonumber
\\&=\int_{\pi(\Omega)}|\phi_{\mathbf a}(\zeta)|^2\|z^{\mathbf a}\|^2_{\Omega_{\zeta}}dV(\zeta),
\end{align}
we have $\phi_{\mathbf a}\in A^2(\pi(\Omega),\|z^{\mathbf a}\|^2_{\Omega_{\zeta}})$. 
Note that $A^2(\pi(\Omega),\|z^{\mathbf a}\|^2_{\Omega_{\zeta}})$ inherits its completeness from $A^2(\Omega)$: Consider an arbitrary compact set $K\subseteq \pi(\Omega)$. Since $\Omega$ is $n$-star-shaped Hartogs, the compact set $\{0\}\times K$ is in $\Omega$. Thus there exists a constant $r_{K}>0$ such that for any point $(0,\zeta)\in\{0\}\times K$, the $(n+m)$-ball $B((0,\zeta);r_K)$ is contained in $\Omega$. Let $r=r_K/3$. Let $B_r^n$ denote the $n$-ball centered at the point $z_r=(\frac{r}{n},\dots,\frac{r}{n})$ with radius $\frac{r}{2n}$. For $\zeta\in K$, let $B_{\zeta}^m$ denote the $m$-ball centered at the $\zeta$ with radius $r$. Then we have $B_r^n\times B_{\zeta}^m\subseteq B((0,\zeta);r_K)\subseteq\Omega$. Let $g(\zeta)$ be an element of $A^2(\pi(\Omega),\|z^{\mathbf a}\|^2_{\Omega_{\zeta}})$. By the mean value property and H\"older inequality,
$$|g(\zeta)|=\Big|\frac{z^{\mathbf a}_rg(\zeta)}{z^{\mathbf a}_r}\Big|\leq\frac{\int_{B_r^n\times B_{\zeta}^m}|z^{\mathbf a}g(w)|dV(z,w)}{Vol(B_r^n\times B_{\zeta}^m)|z_r^{\mathbf a}|}\leq C_K\|g(\zeta)\|_{A^2(\pi(\Omega),\|z^{\mathbf a}\|^2_{\Omega_{\zeta}})}.$$
Taking the supremum of $|g(\zeta)|$ on $K$, we have
$$\sup\limits_{\zeta\in K}|g(\zeta)|\leq C_K\|g(\zeta)\|_{A^2(\pi(\Omega),\|z^{\mathbf a}\|^2_{\Omega_{\zeta}})}.$$
$L^2$ convergence in $A^2(\pi(\Omega),\|z^{\mathbf a}\|^2_{\Omega_{\zeta}})$ implies normal convergence in $A(\pi(\Omega))$, and hence $A^2(\pi(\Omega),\|z^{\mathbf a}\|^2_{\Omega_{\zeta}})$ is closed.

Let $\{\phi_{\mathbf a,\mathbf b}\}$ be a complete orthogonal system of $A^2(\pi(\Omega),\|z^{\mathbf a}\|^2_{\Omega_{\zeta}})$. We finish the proof by showing that $\{z^{\mathbf a}\phi_{\mathbf a,\mathbf b}(\zeta)\}$ forms a complete orthogonal system of $A^2(\Omega)$. For any $f\in A^2(\Omega)$,
$$f(z,\zeta)=\sum_{\mathbf a,\mathbf b}c_{\mathbf a,\mathbf b}z^{\mathbf a}\phi_{\mathbf a,\mathbf b}({\zeta}).$$
 To show the completeness, we assume $f\in A^2(\Omega)$ and  
$\int_{\Omega}f(z,\zeta)\bar{z}^{\mathbf a}\overline{\phi_{\mathbf a,\mathbf b}(\zeta)}dV=0$ for all $\mathbf a,\mathbf b$.  We verify that $f=0$.

Let $\{\Omega^k\}$ be the domains used above. For arbitrary $\mathbf a$ and $\mathbf b$,
\begin{eqnarray}
\int_{\Omega^k}f(z,\zeta)\bar{z}^{\mathbf a}\overline{\phi_{\mathbf a,\mathbf b}(\zeta)}dV+\int_{\Omega-\Omega^k}f(z,\zeta)\bar{z}^{\mathbf a}\overline{\phi_{\mathbf a,\mathbf b}(\zeta)}dV=0\nonumber.
\end{eqnarray}
Taking the absolute value for both terms, we have 
\begin{eqnarray}
\Big|\int_{\Omega^k}f(z,\zeta)\bar{z}^{\mathbf a}\overline{\phi_{\mathbf a,\mathbf b}(\zeta)}dV\Big|=\Big|\int_{\Omega-\Omega^k}f(z,\zeta)\bar{z}^{\mathbf a}\overline{\phi_{\mathbf a,\mathbf b}(\zeta)}dV\Big|\nonumber.
\end{eqnarray}
By H\"older's inequality
\begin{eqnarray}
\Big|\int_{\Omega-\Omega^k}f(z,\zeta)\bar{z}^{\mathbf a}\overline{\phi_{\mathbf a,\mathbf b}(\zeta)}dV\Big|\leq\|z^{\alpha}\phi_{\mathbf a,\mathbf b}(\zeta)\|_{A^2(\Omega)}\Big(\int_{\Omega-\Omega^k}|f|^2dV\Big)^{\frac{1}{2}}\nonumber
\end{eqnarray}
Since $f\in A^2(\Omega)$ and $\Omega^k$ exhausts $\Omega$,
\begin{eqnarray}
\lim\limits_{k\rightarrow\infty}\int_{\Omega-\Omega^k}|f|^2dV= 0\nonumber.
\end{eqnarray}
Therefore
\begin{eqnarray}
\lim\limits_{k\rightarrow\infty}\Big|\int_{\Omega^k}f(z,\zeta)\bar{z}^{\mathbf a}\overline{\phi_{\mathbf a,\mathbf b}(\zeta)}dV\Big|= 0\nonumber.
\end{eqnarray}
Using H\"older's inequality again gives $f(z,\zeta)\bar{z}^{\mathbf a}\overline{\phi_{\mathbf a,\mathbf b}(\zeta)}\in L^1(\Omega)$. The compactness of $\Omega^k$ and polar coordinates give
\begin{align}
&\Big|\int_{\Omega^k}f(z,\zeta)\bar{z}^{\mathbf a}\overline{\phi_{\mathbf a,\mathbf b}(\zeta)}dV\Big|\nonumber
\\=&\Big |\int_{\Omega^k}\sum_{\mathbf i,\mathbf j}c_{\mathbf i,\mathbf j}z^{\mathbf i}\phi_{\mathbf i,\mathbf j}({\zeta})\bar{z}^{\mathbf a}\overline{\phi_{\mathbf a,\mathbf b}(\zeta)}dV\Big|\nonumber
\\=&\Big|\int_{\Omega^k}\sum_{\mathbf j}c_{\mathbf a,\mathbf j}|{z}^{\mathbf a}|^2\phi_{\mathbf a,\mathbf j}({\zeta})\overline{\phi_{\mathbf a,\mathbf b}(\zeta)}dV\Big|\nonumber.
\end{align}
By the Dominated Convergence Theorem,
\begin{align}
&\lim\limits_{k\rightarrow\infty}\Big|\int_{\Omega^k}\sum_{\mathbf j}c_{\mathbf a,\mathbf j}|{z}^{\mathbf a}|^2\phi_{\mathbf a,\mathbf j}({\zeta})\overline{\phi_{\mathbf a,\mathbf b}(\zeta)}dV\Big|\nonumber
\\=&\Big|\int_{\Omega}\lim\limits_{k\rightarrow\infty}\chi_{\Omega^k}(z,\zeta)\sum_{\mathbf j}c_{\mathbf a,\mathbf j}|{z}^{\mathbf a}|^2\phi_{\mathbf a,\mathbf j}({\zeta})\overline{\phi_{\mathbf a,\mathbf b}(\zeta)}dV\Big|\nonumber
\\=&\Big|\int_{\Omega}\sum_{\mathbf j}c_{\mathbf a,\mathbf j}|{z}^{\mathbf a}|^2\phi_{\mathbf a,\mathbf j}({\zeta})\overline{\phi_{\mathbf a,\mathbf b}(\zeta)}dV\Big|\nonumber
\\=&\Big|\int_{\Omega}c_{\mathbf a,\mathbf b}|{z}^{\mathbf a}|^2\phi_{\mathbf a,\mathbf b}({\zeta})\overline{\phi_{\mathbf a,\mathbf b}(\zeta)}dV\Big|\nonumber
\\=&c_{\mathbf a,\mathbf b}\|{z}^{\mathbf a}\phi_{\mathbf a,\mathbf b}({\zeta})\|^2_{A^2(\Omega)}\nonumber.
\end{align}
Therefore $c_{\mathbf a,\mathbf b}=0$ for all $\mathbf a$, $\mathbf b$ and $f\equiv0$. 
\end{proof}
\begin{re*}
Since $\Omega$ is not necessarily bounded in Lemma 2.2, it is possible that the weighted space $A^2(\pi(\Omega),\|z^{\mathbf a}\|^2_{\Omega_{\zeta}})$ is trivial and $\phi_{\mathbf a,\mathbf b}=0$. Nevertheless, the second statement of Lemma 2.2 remains true.
\end{re*}
\begin{co}
	Let $\phi_{\mathbf a,\mathbf b}(\zeta)$ be a complete orthogonal system for $A^2(D,\|z^{\mathbf a}\|^2_{\Omega_{\zeta}})$. Then
\begin{equation}\label{7}
K_{\Omega}(z,z^{\prime};\bar\zeta,\bar{\zeta^{\prime}})=\sum_{\mathbf a,\mathbf b}\frac{(z\bar \zeta)^{\mathbf a}\phi_{\mathbf a,\mathbf b}(z^{\prime})\overline{\phi_{\mathbf a,\mathbf b}(\zeta^{\prime})}}{\|z^{\mathbf a}\phi_{\mathbf a,\mathbf b}(z^{\prime})\|^2_{L^2(\Omega)}}.
\end{equation}
\end{co}
 When $m=0$, $\Omega$ becomes a Reinhardt domain containing the origin and (\ref{7}) becomes (\ref{1}).

\section{Main Results}
Let $\Omega\subseteq \mathbb C^{n+m}$ be an $n$-star-shaped Hartogs domain in the first $n$ variables $(z_1,\dots,z_n)$.
Our technique for computing the Bergman kernel function works on the following two kinds of domains:
\begin{itemize}
	\item
	$U^{\alpha}=\big\{(z,z^{\prime},w)\in \mathbb C^{n+m}\times \mathbb C:\big(f_{\alpha}(z,w),z^{\prime}\big)\in \Omega,|w|<1\big\}$
	\\where  $$f_{\alpha}(z,w)=\Big(\frac{z_1}{(1-|w|^2)^{\frac{\alpha_1}{2}}},\dots,\frac{z_n}{(1-|w|^2)^{\frac{\alpha_n}{2}}}\Big)$$
	and $\alpha_j$'s are positive numbers.
	\item
	$V^{\gamma}=\big\{(z,z^{\prime},w)\in \mathbb C^{n+m}\times \mathbb C:\big(g_{\gamma}(z,w),z^{\prime}\big)\in \Omega\big\}$
	\\where
	$$g_{\gamma}(z,w)=\Big(e^{\frac{\gamma_1|w|^{2}}{2}}z_1,\dots,e^{\frac{\gamma_n|w|^{2}}{2}}z_n\Big)$$
	and $\gamma_j$'s are positive numbers.
\end{itemize}	
\begin{re*} In our definition, we avoid the cases when all $\alpha_j$'s and $\gamma_j$'s equal 0 since they are not interesting. When $\alpha=\mathbf 0$, $U^{\mathbf 0}$ becomes $\Omega\times \mathbb B^1$ and $K_{U^{\mathbf 0}}$ equals the product of the Bergman kernels on $\Omega$ and the unit disk $\mathbb B^1$. When $\gamma=\mathbf 0$, $V^{\mathbf 0}=\Omega\times \mathbb C$. Since $A^2(V^{\mathbf 0})=\{0\}$, the kernel function $K_{V^{\mathbf 0}}$ is identically zero. These results are consistent with Theorems 1 and 2.
\end{re*}
Since $e^{|w|^2}$ and $(1-|w|^2)^{-1}$ are increasing in $|w|$ and invariant under the rotation map $w\mapsto e^{i\theta}w$ for $\theta\in \mathbb R$, the slice domains of $U^{\alpha}$ and $V^\gamma$ with $z$ and $z^{\prime}$ coordinates fixed are disks in $\mathbb C$. This observation yields the following:
\begin{lem} If $\Omega$ is n-star-shaped Hartogs in the variables $(z_1,\dots,z_n)$, then
 	$U^{\alpha}$ and $V^{\gamma}$ are $(n+1)$-star-shaped Hartogs in the variables $(z_1,\dots,z_n,w)$.
 \end{lem}
By Lemma 2.2, a complete orthogonal system of the form $\{z^{\mathbf a}\phi_{\mathbf a,\mathbf b,c}(z^{\prime})w^c\}$ can be chosen for $A^2(U^{\alpha})$ and $A^2(V^\gamma)$. The next lemma implies that $\{z^{\mathbf a}\phi_{\mathbf a,\mathbf b}(z^{\prime})w^c\}$ is a complete orthogonal system for both $A^2(U^{\alpha})$ and $A^2(V^\gamma)$ if $\{z^{\mathbf a}\phi_{\mathbf a,\mathbf b}(z^{\prime})\}$ is a complete orthogonal system for $A^2(\Omega)$.
\begin{lem}
 	The function $z^{\mathbf a}\phi({z^{\prime}})$ is square-integrable on $\Omega$ if and only if for all $c\in \mathbb N$, the function $z^{\mathbf a}\phi({z^{\prime}})w^c$ is square-integrable on $U^{\alpha}$ $\big(\text {or}\:\:V^{\gamma}\big)$.
 \end{lem}
 \begin{proof}
Suppose $z^{\mathbf a}\phi({z^{\prime}})w^c\in A^2(U^{\alpha})$. Then 
\begin{equation}\label{8}
\int_{U^\alpha}|z^{\mathbf a}|^2|\phi({z^{\prime}})|^2|w|^{2c}dV(z,z^\prime,w)=\|z^{\mathbf a}\phi({z^{\prime}})w^c\|^2_{L^2(U^{\alpha})}<\infty.
\end{equation}
Substituting $t_j={z_j}{(1-|w|^2)^{-\frac{\alpha_j}{2}}}$ for $1\leq j\leq n$ and applying Fubini's theorem to the integral in (\ref{8}) yield:
\begin{align}\label{a}
 &\int_{U^\alpha}|z^{\mathbf a}|^2|\phi({z^{\prime}})|^2|w|^{2c}dV(z,z^\prime,w)\nonumber
 \\=&\int_{\mathbb B^1}|w|^{2c}(1-|w|^2)^{\alpha\cdot(\mathbf a+\mathbf 1)}dV(w)\int_{\Omega}|t^{\mathbf a}|^2|\phi(z^{\prime})|^2dV(t,z^{\prime})
 \nonumber
  \\=&\int_{\mathbb B^1}|w|^{2c}(1-|w|^2)^{\alpha\cdot(\mathbf a+\mathbf 1)}dV(w)\big\|z^{\mathbf a}\phi({z^{\prime}})\big\|^2_{L^2(\Omega)}<\infty.
\end{align}
Since $\int_{\mathbb B^1}|w|^{2c}(1-|w|^2)^{\alpha\cdot(\mathbf a+\mathbf 1)}dV(w)$ is a constant, $\|z^{\mathbf a}\phi({z^{\prime}})\|^2_{L^2(\Omega)}<\infty$ and $z^{\mathbf a}\phi({z^{\prime}})$ is in $A^2(\Omega)$. By (\ref{a}), the converse is also true. A similar argument proves the statement for $V^{\gamma}$. We omit the details.
 \end{proof}
The definitions of $U^\alpha$ and $V^{\gamma}$ also imply that the slices of $U^\alpha$ and $V^{\gamma}$, with the $w$ coordinate fixed, are biholomorphic to $\Omega$.
  For fixed $w\in \mathbb B^1$ and $\eta\in \mathbb C$, let $U^{\alpha}_w$ denote the slice domain $\{(z,z^{\prime})\in \mathbb C^{n+m}:(z,z^{\prime},w)\in U^{\alpha}\}$ of $U^{\alpha}$ and let $V^{\gamma}_\eta$  denote the slice domain $\{(z,z^{\prime})\in \mathbb C^{n+m}:(z,z^{\prime},\eta)\in V^{\gamma}\}$ of $V^{\gamma}$. Applying the mappings $f_\alpha(\cdot,w)$ and $g_\gamma(\cdot,\eta)$ to $U^{\alpha}_w$ and $V^{\gamma}_\eta$ yields:
\begin{lem}
 $U^{\alpha}_w$ and $V^{\gamma}_\eta$ are biholomorphic to $\Omega$.
 \end{lem}
Computing the Levi form of $U^{\alpha}$ and $V^{\gamma}$ yields the following:
\begin{re*} For a smooth and pseudoconvex $\Omega$, both $U^{\alpha}$ and $V^{\gamma}$ are pseudoconvex.\end{re*}
As we will see in Sections 5 and 6, $U^{\alpha}$ and $V^{\gamma}$ may not be strongly pseudoconvex even if $\Omega$ is strongly pseudoconvex.

We denote by $K_{U^{\alpha}}$ and $K_{V^{\gamma}}$ the Bergman kernel functions on $U^{\alpha}$ and $V^{\gamma}$. We denote by $K_{U^{\alpha}_w}$ and $K_{V^{\gamma}_w}$ the Bergman kernel functions on $U^{\alpha}_w$ and $V^{\gamma}_w$. We obtain $K_{U^{\alpha}}$ and $K_{V^{\gamma}}$ from $K_{U^{\alpha}_w}$ and $K_{V^{\gamma}_w}$ by the following procedure as mentioned in the introduction:
\begin{itemize}
\item[\textbf {Step} {\bf 1.}]First we evaluate $K_{U^{\alpha}_w}$ and $K_{V^{\gamma}_w}$ at appropriate points (off the diagonal).
\item[\textbf {Step} {\bf 2.}]With further modification, we obtain Hermitian symmetric functions on ${U^{\alpha}}$ and ${V^{\gamma}}$. \item[\textbf {Step} {\bf 3.}]By applying a first order differential operator to the result in \textbf{Step} $\bf 2$, we obtain the Bergman kernel functions on the target domains. 
\end{itemize}
The modifying functions $h$ and $l$ for $K_{U^{\alpha}_{\eta}}$ and $K_{{V^{\gamma}_{\eta}}}$ in \textbf{Step $\mathbf 1$} are defined as follows:
\begin{align}
&h(z,w,\eta)=\Big(z_1(\frac{1-|\eta|^2}{1-w\bar \eta})^{\alpha_1},\dots,z_n(\frac{1-|\eta|^2}{1-w\bar \eta})^{\alpha_n}\Big),\tag{$\rm i$}\\
&l(z,w,\eta)=\Big(z_1e^{\gamma_1(w\bar{\eta}-|\eta|^2)},\dots,z_ne^{\gamma_n(w\bar{\eta}-|\eta|^2)}\Big).\tag{$\rm ii$}
\end{align}
The modification in \textbf{Step $\bf 2$} only involves multiplication. For simplicity, we let operators $D_{U^{\alpha}}$ and $D_{V^{\gamma}}$ denote the composition of the multiplication operator in {\textbf{Step $\bf 2$}} and the differential operator in {\textbf{Step $\bf 3$}} for $U^{\alpha}$ and $V^{\gamma}$.  $D_{U^{\alpha}}$ and $D_{V^{\gamma}}$ are defined by 
\begin{align}
&D_{U^{\alpha}}=\frac{(1-|\eta|^2)^{\alpha\cdot\mathbf 1}}{\pi(1-w\bar \eta)^{2+\alpha\cdot\mathbf 1}}\Big(I+\sum_{j=1}^{n}\alpha_j(I+z_j\frac{\partial}{\partial z_j})\Big),\tag{$*$}\\
&D_{V^{\gamma}}=\frac{e^{(\gamma\cdot\mathbf 1)(w\bar{\eta}-|\eta|^{2})}}{\pi}\Big(\sum_{j=1}^{n}\gamma_j(I +z_j\frac{\partial}{\partial z_j})\Big).\tag{$**$}
\end{align}
Here are the main results:
\begin{thm}Let $U^{\alpha}$ and $U^{\alpha}_\eta$ be defined as above. For $(z,z^{\prime},w;\zeta,\zeta^{\prime},\eta)\in U^{\alpha}\times U^{\alpha}$, let $h(z,w,\eta)$ and $D_{U^{\alpha}}$ be as in $(\rm i)$ and $(*)$. Then
	\begin{equation}\label{d}
	K_{U^{\alpha}}(z,z^{\prime},w;\bar \zeta,\bar \zeta^{\prime},\bar \eta)=D_{U^{\alpha}}K_{U^{\alpha}_{\eta}}\big(h(z,w,\eta),z^{\prime};\bar\zeta,\bar{\zeta^{\prime}}\big).
	\end{equation}
\end{thm}
\begin{thm}
Let $V^{\gamma}$ and $V^{\gamma}_{\eta}$ be defined as above. For $(z,z^{\prime},w;\zeta,\zeta^{\prime},\eta) \in V^{\gamma}\times V^{\gamma}$, let  $l(z,w,\eta)$ and $D_{V^{\gamma}}$ be as in $(\rm ii)$ and $(**)$. Then
		\begin{equation}\label{9}
		K_{V^{\gamma}}(z,z^{\prime},w;\bar \zeta,\bar \zeta^{\prime},\bar \eta)=D_{V^{\gamma}}K_{V^{\gamma}_{\eta}}\big(l(z,w,\eta),z^{\prime};\bar\zeta,\bar{\zeta^{\prime}}\big).
		\end{equation}
\end{thm}

Theorems 1 and 2 are proved in a similar way. 
We illustrate using Theorem 1. We first show that the function on the right hand side of (\ref{d}) is defined on $U^{\alpha}\times U^{\alpha}$. Then we prove, for the complete orthogonal system $\{z^{\mathbf a}\phi_{\mathbf a,\mathbf b}(z^{\prime})\}$ of $A^2(\Omega)$, that the function has the following expansion:
\begin{equation}\label{10}
\sum_{\mathbf a,\mathbf b,c}c_{\mathbf a,\mathbf b,c}(z\bar \zeta)^{\mathbf a}\phi_{\mathbf a,\mathbf b}(z^{\prime})\overline{\phi_{\mathbf a,\mathbf b}(\zeta^{\prime})}(w\bar{\eta})^c.
\end{equation}
By showing that (\ref{10}) reproduces every element in $A^2(U^{\alpha})$, we conclude that the equality in (\ref{d}) holds and our proof is complete.
In the proof, we let $\Gamma$ denote the gamma function and let $(a)_b$  denote the Pochhammer symbol $\frac{\Gamma(a+b)}{\Gamma(a)}$.
\begin{proof}[Proof of Theorem 1]
Let $K_1(z,z^{\prime},w;\bar \zeta,\bar{\zeta^{\prime}},\bar \eta)$ denote
\begin{equation}
D_{U^{\alpha}}K_{U^{\alpha}_{\eta}}\big(h(z,w,\eta),z^{\prime};\bar\zeta,\bar{\zeta^{\prime}}\big).\end{equation}
We first show $K_1$ is defined on $U^{\alpha}\times U^{\alpha}$, i.e. $(\zeta,\zeta^{\prime})\in U^{\alpha}_\eta$ and $(h(z,w,\eta),z^{\prime})\in U^{\alpha}_{\eta}$. 
The definition of $U^{\alpha}_\eta$ implies that $(\zeta,\zeta^{\prime})\in U^{\alpha}_\eta$. To prove $(h(z,w,\eta),z^{\prime})\in U^{\alpha}_{\eta}$, it suffices to show $(f_\alpha(h(z,w,\eta),\eta),z^{\prime})\in \Omega$. Note that 
$$\Big(f_\alpha\big(h(z,w,\eta),\eta\big),z^{\prime}\Big)=\Big(\frac{(1-|\eta|^2)^{\frac{\alpha_1}{2}}}{(1-w\bar{\eta})^{\alpha_1}}z_1,\dots,\frac{(1-|\eta|^2)^{\frac{\alpha_n}{2}}}{(1-w\bar{\eta})^{\alpha_n}} z_n,z^{\prime}\Big),$$
and for each $0\leq j\leq n$,
$$\Big|\frac{(1-|\eta|^2)^{\frac{\alpha_j}{2}}}{(1-w\bar{\eta})^{\alpha_j}}z_j\Big|\leq\frac{|z_j|}{(1-|w|^2)^{\frac{\alpha_j}{2}}}.$$
By the fact that $\Omega$ is $n$-star-shaped Hartogs in the first $n$ variables, the containment  $(f_{\alpha}(z,w),z^{\prime})\in \Omega$ implies $(f_\alpha(h(z,w,\eta),\eta),z^{\prime})\in \Omega$.
Therefore $(\zeta,\zeta^{\prime})\in U^{\alpha}_\eta$ and $K_1$ is defined on $U^{\alpha}\times U^{\alpha}$.

\:\:
To show $K_1$ satisfies (\ref{10}), we consider the biholomorphic map $f_{\alpha}(\cdot,\eta)$ from $U^{\alpha}_\eta$ to $\Omega$:
$$f_{\alpha}(z,\eta)=((1-|\eta|^2)^{-\frac{\alpha_1}{2}}z_1,\dots,(1-|\eta|^2)^{-\frac{\alpha_n}{2}}z_n,z^{\prime}).$$
 By (\ref{4}), we have:
\begin{equation}\label{12}
K_{U^{\alpha}_\eta}(z,z^{\prime},w;\bar{\zeta},\bar{\zeta^{\prime}},\bar{\eta})=(1-|\eta|^2)^{-\alpha\cdot \mathbf 1}K_{\Omega}\big(f_{\alpha}(z,\eta),z^{\prime};\overline {f_{\alpha}(\zeta,\eta)},\bar{\zeta^{\prime}}\big).
\end{equation}
Therefore
\begin{equation}\label{13}
K_{1}(z,z^{\prime},w;\bar{\zeta},\bar{\zeta^{\prime}},\bar{\eta})=D_1K_{\Omega}\big(A_1(z,w,\eta),z^{\prime};\overline {f_{\alpha}(\zeta,\eta)},\bar{\zeta^{\prime}}\big),
\end{equation}
where
\begin{align*}
&D_{1}={\pi^{-1}(1-w\bar \eta)^{-(2+\alpha\cdot\mathbf 1)}}\Big(I+\sum_{j=1}^{n}\alpha_j\Big(I+z_j\frac{\partial}{\partial z_j}\Big)\Big),\\
&A_1(z,w,\eta)=\Big(z_1\Big(\frac{\sqrt{1-|\eta|^2}}{1-w\bar \eta}\Big)^{\alpha_1},\dots,z_n\Big(\frac{\sqrt{1-|\eta|^2}}{1-w\bar \eta}\Big)^{\alpha_n}\Big).
\end{align*}
Applying (\ref{7}) to $K_{\Omega}$ yields
\begin{equation}
K_{\Omega}\big(A_1(z,w,\eta),z^{\prime};\overline {f_{\alpha}(\zeta,\eta)},\bar{\zeta^{\prime}}\big)=\sum_{\mathbf a,\mathbf b}\frac{(z\bar \zeta)^{\mathbf a}\phi_{\mathbf a,\mathbf b}(z^{\prime})\overline{\phi_{\mathbf a,\mathbf b}(\zeta^{\prime})}}{(1-w\bar\eta)^{\mathbf a \cdot \alpha}\|z^{\mathbf a}\phi_{\mathbf a,\mathbf b}(z^{\prime})\|^2_{L^2(\Omega)}}\nonumber
.
\end{equation}
Thus, $K_{1}(z,z^{\prime},w;\bar{\zeta},\bar{\zeta^{\prime}},\bar{\eta})$ can be written as
\begin{equation*}
\sum_{\mathbf a,\mathbf b}c_{\mathbf a,\mathbf b,c}{(z\bar \zeta)^{\mathbf a}\phi_{\mathbf a,\mathbf b}(z^{\prime})\overline{\phi_{\mathbf a,\mathbf b}(\zeta^{\prime})}}(w\bar{\eta})^c.
\end{equation*}

We complete the proof by showing that $K_1$ reproduces every element in $A^2(U^{\alpha})$. For arbitrary $z^\mathbf a\phi_{\mathbf a,\mathbf b}(z^\prime)w^c\in A^2(U^{\alpha})$, we consider the integral:
\begin{align}\label{i}
&\int_{U^{\alpha}}K_1(z,z^{\prime},w;\bar{\zeta},\bar{\zeta^{\prime}},\bar{\eta})\zeta^\mathbf a\phi_{\mathbf a,\mathbf b}(\zeta^\prime)\eta^cdV.
\end{align}
By the definitions of $K_1$ and $U^{\alpha}$, (\ref{i}) equals
\begin{align}\label{14}
&\int_{\mathbb B^1}\eta^c\int_{U^{\alpha}_\eta}D_{U^\alpha}K_{U^{\alpha}_{\eta}}\big(h(z,w,\eta),z^{\prime};\bar{\zeta};\bar{\zeta^{\prime}}\big)\zeta^\mathbf a\phi_{\mathbf a,\mathbf b}(\zeta^\prime)dV(\zeta,\zeta^{\prime})dV(\eta).
\end{align}
Using the reproducing property of $K_{U^{\alpha}_\eta}$ on $U^{\alpha}_\eta$ and Corollary 2.1, we have
\begin{align}
&\int_{U^{\alpha}_\eta}D_{U^\alpha}K_{U^{\alpha}_{\eta}}(h(z,w,\eta),z^{\prime};\bar{\zeta},\bar{\zeta^{\prime}})\zeta^\mathbf a\phi_{\mathbf a,\mathbf b}(\zeta^\prime)dV(\zeta,\zeta^{\prime})\nonumber
\\=&(1+\alpha\cdot(\mathbf a +\mathbf 1))\frac{(1-|\eta|^2)^{\alpha\cdot\mathbf 1}}{\pi(1-w\bar\eta)^{2+\alpha\cdot\mathbf 1}}h(z,w,\eta)^{\mathbf a}\phi_{\mathbf a,
\mathbf b}(z^{\prime}).
\end{align}
Therefore (\ref{14}) becomes
\begin{equation}\label{16}
\big(1+\alpha\cdot(\mathbf a +\mathbf 1)\big)\phi_{\mathbf a,\mathbf b}(z^\prime)\int_{\mathbb B^1}\frac{(1-|\eta|^2)^{\alpha\cdot\mathbf 1}\eta^ch(z,w,\eta)^{\mathbf a}}{\pi(1-w\bar\eta)^{2+\alpha\cdot \mathbf 1}}dV(\eta).
\end{equation}
Since
$h(z,w,\eta)=(z_1(\frac{1-|\eta|^2}{1-w\bar \eta})^{\alpha_1},\dots,z_n(\frac{1-|\eta|^2}{1-w\bar \eta})^{\alpha_n})$, (\ref{16}) equals
\begin{equation}\label{17}
\big(1+\alpha\cdot(\mathbf a +\mathbf 1)\big)z^{\mathbf a}\phi_{\mathbf a,\mathbf b}(z^\prime)\int_{\mathbb B^1}\frac{(1-|\eta|^2)^{\alpha\cdot(\mathbf a+\mathbf 1)}\eta^c}{\pi(1-w\bar\eta)^{2+\alpha\cdot (\mathbf a+\mathbf 1)}}dV(\eta).
\end{equation}
Expanding the denominator in (\ref{17}) yields
\begin{align}\label{h}
(\ref{17})=&z^{\mathbf a}\phi_{\mathbf a,\mathbf b}(z^\prime)\int_{\mathbb B^1}\sum_{j=0}^{\infty}\frac{(1+\alpha\cdot (\mathbf a+\mathbf 1))_{j+1}(1-|\eta|^2)^{\alpha\cdot(\mathbf a+\mathbf 1)}(w\bar{\eta})^j}{\pi j!}\eta^cdV(\eta)\nonumber
\\=&z^{\mathbf a}\phi_{\mathbf a,\mathbf b}(z^\prime)w^c\int_{\mathbb B^1}\frac{(1+\alpha\cdot (\mathbf a+\mathbf 1))_{c+1}(1-|\eta|^2)^{\alpha\cdot(\mathbf a+\mathbf 1)}|{\eta}|^{2c}}{\pi c!}dV(\eta).
\end{align}
By substituting $r=|\eta|^2$ to the last line of (\ref{h}), we obtain
\begin{align}
&z^{\mathbf a}\phi_{\mathbf a,\mathbf b}(z^\prime)w^c\int_{0}^{1}\frac{\big(1+\alpha\cdot (\mathbf a+\mathbf 1)\big)_{c+1}(1-r)^{\alpha\cdot(\mathbf a+\mathbf 1)}r^{c}}{ c!}dr\nonumber
\\=&z^{\mathbf a}\phi_{\mathbf a,\mathbf b}(z^\prime)w^c\frac{\big(1+\alpha\cdot (\mathbf a+\mathbf 1)\big)_{c+1}}{ c!}\frac{\Gamma\big({1+\alpha\cdot(\mathbf a+\mathbf 1)\big)\Gamma{(c+1)}}}{\Gamma\big({2+\alpha\cdot(\mathbf a+\mathbf 1)+c}\big)}\nonumber
\\=&z^{\mathbf a}\phi_{\mathbf a,\mathbf b}(z^\prime)w^c\nonumber.
\end{align}
Thus the reproducing property holds, and $K_1$ is the Bergman kernel.
\end{proof}	
\begin{proof}[Proof of Theorem 2]
Let $K_2$ denote the right hand side of (\ref{9}). By the same argument in the proof of Theorem 1, we can prove that $K_2$ is defined on $V^{\gamma}\times V^{\gamma}$ and can be written as:
\begin{equation*}
\sum_{\mathbf a,\mathbf b}c_{\mathbf a,\mathbf b,c}{(z\bar \zeta)^{\mathbf a}\phi_{\mathbf a,\mathbf b}(z^{\prime})\overline{\phi_{\mathbf a,\mathbf b}(\zeta^{\prime})}}(w\bar{\eta})^c.
\end{equation*}
We show that $K_2$ reproduces every element in $A^2(V^{\gamma})$. For $z^\mathbf a\phi_{\mathbf a,\mathbf b}(z^\prime)w^c\in A^2(V^{\gamma})$, we consider the integral
\begin{align}\label{g}
&\int_{V^\gamma}K_2(z,z^{\prime},w;\bar{\zeta},\bar{\zeta^{\prime}},\bar{\eta})\zeta^\mathbf a\phi_{\mathbf a,\mathbf b}(\zeta^\prime)\eta^cdV.
\end{align}
By the definitions of $K_2$ and $V^{\gamma}$, (\ref{g}) equals
\begin{align}\label{19}
&\int_{\mathbb C}\eta^c\int_{V^\gamma_\eta}D_{V^\alpha}K_{V^{\gamma}_{\eta}}\big(l(z,w,\eta),z^{\prime};\bar{\zeta},\bar{\zeta^{\prime}}\big)\zeta^\mathbf a\phi_{\mathbf a,\mathbf b}(\zeta^\prime)dV(\zeta,\zeta^{\prime})dV(\eta).
\end{align}
By the reproducing property of $K_{V^{\gamma}_\eta}$ and Corollary 2.1, we have
\begin{align}\label{f}
&\int_{V^\gamma_\eta}D_{V^\alpha}K_{V^{\gamma}_{\eta}}\big(l(z,w,\eta),z^{\prime};\bar{\zeta},\bar{\zeta^{\prime}}\big)\zeta^\mathbf a\phi_{\mathbf a,\mathbf b}(\zeta^\prime)dV(\zeta,\zeta^{\prime})\nonumber
\\=&\pi^{-1}\big(\gamma\cdot(\mathbf a+\mathbf 1)\big)\phi_{\mathbf a,\mathbf b}(z^{\prime})e^{(\gamma\cdot(\mathbf a+\mathbf 1))(w\bar\eta-|\eta|^2)}z^{\mathbf a}.
\end{align}
Substituting (\ref{f}) to (\ref{19}) yields
\begin{equation}\label{21}
\pi^{-1}\big(\gamma\cdot(\mathbf a +\mathbf 1)\big)z^{\mathbf a}\phi_{\mathbf a,\mathbf b}(z^\prime)\int_{\mathbb C}\frac{e^{\gamma\cdot(\mathbf a+\mathbf 1)w\bar{\eta}}\eta^c}{e^{(\gamma\cdot(\mathbf a+\mathbf 1))|\eta|^2}}dV(\eta).
\end{equation}
Expanding $e^{\gamma\cdot(\mathbf a+\mathbf 1)w\bar{\eta}}$ in (\ref{21}), we have
\begin{align}\label{22}
(\ref{21})=&z^{\mathbf a}\phi_{\mathbf a,\mathbf b}(z^\prime)\int_{\mathbb C}\sum_{j=0}^{\infty}\frac{\big(\gamma\cdot (\mathbf a+\mathbf 1)\big)^{j+1}(w\bar{\eta})^j}{\pi j!}\eta^ce^{-\gamma\cdot(\mathbf a+\mathbf 1)|\eta|^2}dV(\eta)\nonumber
\\=&z^{\mathbf a}\phi_{\mathbf a,\mathbf b}(z^\prime)w^c\int_{\mathbb C}\frac{\big(\gamma\cdot (\mathbf a+\mathbf 1)\big)^{c+1}|{\eta}|^{2c}}{\pi c!}e^{-\gamma\cdot(\mathbf a+\mathbf 1)|\eta|^2}dV(\eta).
\end{align}
Letting $t=\gamma\cdot(\mathbf a+\mathbf 1)|\eta|^2$ and using polar coordinates, the last line of (\ref{22}) becomes
\begin{align}
&z^{\mathbf a}\phi_{\mathbf a,\mathbf b}(z^\prime)w^c\int_{0}^{\infty}\frac{t^{c}}{c!}e^{-t}dt,
\end{align}
which equals $z^{\mathbf a}\phi_{\mathbf a,\mathbf b}(z^\prime)w^c\nonumber.$
Therefore $K_2$ is the Bergman kernel on $V^{\gamma}$.
\end{proof}
When the ``base" domain $\Omega$ is $n$-star-shaped Hartogs, we can apply Theorems 1 and 2 to obtain the kernel functions on the domains $U^{\alpha}$ and $V^{\gamma}$. Lemma 3.1 therefore enables us to apply Theorem 1 and 2 several times to obtain the Bergman kernel on more complicated domains. Moreover, we can obtain the Bergman kernel when the $w$ in $U^{\alpha}$ and $V^{\gamma}$ is a vector instead of a single variable. We'll discuss these refinements in Section 7.  

\section{Examples}
Theorems 1 and 2 enable us to explicitly compute the Bergman kernel in several new situations. First we combine Theorem 1 with the inflation method in \cite{1999BFS} to give a new proof of the explicit formula in \cite{1994D'A2}. Then we compute the kernel function in two new cases.
\begin{ex}
Let the ``base" domain $\Omega$ be the unit ball $\mathbb B^n$ in $\mathbb C^n$. For $p>0$, put $\alpha=({\frac{1}{p}},\dots,{\frac{1}{p}})$. We have
\begin{displaymath}
U^{\alpha}=\{(z,w)\in \mathbb C^n\times \mathbb C:\|z\|^{2p}+|w|^2<1\}.
\end{displaymath}
By Theorem 1, the Bergman kernel function $K_{U^{\alpha}}$ is equal to:
\begin{equation*}
\frac{n!}{\pi^{n+1}p}\frac{(n+p)(1-w\bar{\eta})^{\frac{1}{p}}+(1-p)\langle z,\zeta\rangle}{(1-w\bar{\eta})^{2-\frac{1}{p}}((1-w\bar{\eta})^{\frac{1}{p}}-\langle z,\zeta\rangle)^{n+2}}.
\end{equation*}
\end{ex}
Let $U^{\alpha\prime}=\{(z,w)\in \mathbb C^n\times \mathbb C^m:\|z\|^{2p}+\|w\|^2<1\}$.
Applying the inflation method to $K_{U^{\alpha}}$ yields the Bergman kernel function on $U^{\alpha\prime}$:
\begin{equation*}
\frac{n!}{\pi^{m+n}p}\Big(\frac{\partial}{\partial t}\Big)^{m-1}\frac{(n+p)(1-\langle w,\eta\rangle)^{\frac{1}{p}}+(1-p)\langle z,\zeta\rangle}{(1-\langle w,\eta\rangle)^{2-\frac{1}{p}}((1-\langle w,\eta\rangle)^{\frac{1}{p}}-\langle z,\zeta\rangle)^{n+2}}
\end{equation*}
where $t=\langle w,\eta \rangle$.

Note that if we let the above $p$ tend to $\infty$, then $U^{\alpha}$ becomes $\mathbb B^n\times \mathbb B^1$ and the Bergman kernel $K_{U^{\alpha}}$ equals $K_{\mathbb B^n}\cdot K_{\mathbb B^1}$.
\begin{ex}
	Suppose $\Omega=\{(z,z^{\prime})\in \mathbb C^n\times \mathbb C^m:\|z\|^2+\|z^{\prime}\|^2<1\}$ and $\alpha=(1,\cdots,1)$, then
	\begin{displaymath}
		U^{\alpha}=\{(z,z^{\prime},w)\in \mathbb C^n\times\mathbb C^m\times \mathbb C:|w|<1\:\:\text{and}\:\:\|z\|^2+\|z^{\prime}\|^2+|w|^2<1+|w|^2\|z^{\prime}\|^2\}
	\end{displaymath}
	has the Bergman kernel function:
	\begin{equation}\label{23}
		K_{U^{\alpha}}=\frac{(m+n)!}{\pi^{m+n+1}}\frac{(1-w\bar{\eta})^m(n+1-(n+1)\langle z^{\prime},\zeta^{\prime}\rangle+m\frac{\langle z,\zeta\rangle}{1-w\bar{\eta}})}{(1-w\bar{\eta}-\langle z,\zeta\rangle-\langle z^{\prime},\zeta^{\prime}\rangle+w\bar{\eta}\langle z^{\prime},\zeta^{\prime}\rangle)^{m+n+2}}.
	\end{equation}
	When $m=0$, the right hand side of (\ref{23}) becomes
	\begin{equation*}
		\frac{n+1}{\pi}\frac{n!}{\pi^{n}}\frac{1}{(1-w\bar{\eta}-\langle z,\zeta\rangle)^{n+2}},
	\end{equation*}
which is the Bergman kernel function on the unit ball $\mathbb B^{n+1}$.
\end{ex}
When $n=m=1$, $U^{\alpha}$ becomes 
\begin{equation*}
\{(z,z^{\prime},w)\in \mathbb C^3:|w|<1,|z|^2+|z^{\prime}|^2+|w|^2<1+|w|^2|z^{\prime}|^2\},
\end{equation*}
which is mentioned in the abstract.
Using (\ref{23}), we have:
\begin{equation*}
	K_{U^{\alpha}}=\frac{2}{\pi^{3}}\frac{(1-w\bar{\eta})(2-2z^{\prime}\bar{\zeta^{\prime}}+\frac{ z\bar{\zeta}}{1-w\bar{\eta}})}{(1-w\bar{\eta}-z\bar{\zeta}- z^{\prime}\bar{\zeta^{\prime}}+w\bar{\eta} z^{\prime}\bar{\zeta^{\prime}})^{4}}.
\end{equation*}
\begin{ex}
	Let $\Omega=\{(z,z^{\prime})\in \mathbb C^n\times \mathbb C^m:\|z\|^2+\|z^{\prime}\|^2<1\}$ and $\gamma=(\gamma_1,\dots,\gamma_n)$, then 
	\begin{equation*}
	V^{\gamma}=\{(z,z^{\prime},w)\in \mathbb C^n\times\mathbb C^m\times\mathbb C;\sum_{j=1}^{n}e^{\gamma_j|w|^2}|z_j|^2+\|z^{\prime}\|^2<1\}.
	\end{equation*}
	Put $\rho(z,z^{\prime},w;\bar\zeta,\bar{\zeta^{\prime}},\bar\eta)=1-\sum_{j=1}^{n}e^{\gamma_jw\bar \eta}z_j\bar\zeta_j-\langle z^{\prime},\zeta^{\prime}\rangle$. Then the Bergman kernel function $K_{V^{\gamma}}$ equals
	\begin{equation}\label{24}
\frac{(m+n)!e^{(\gamma\cdot \mathbf 1)w\bar{\eta}}}{\pi^{m+n+1}}\Big(\frac{\gamma\cdot \mathbf 1}{\rho^{m+n+1}}+
	\frac{(m+n+1)\sum_{j=1}^{n}\gamma_je^{\gamma_jw\bar \eta}z_j\bar\zeta_j}{\rho^{m+n+2}}\Big).
	\end{equation}
\end{ex}

When $\gamma=\mathbf 1$ and $n=m=1$, $V^{\gamma}$ becomes
\begin{equation}
\{(z,z^{\prime},w)\in \mathbb C^3:e^{|w|^2}|z|^2+|z^{\prime}|^2<1\},
\end{equation}
which is mentioned in the abstract. Using (\ref{24}), we obtain its kernel function:
\begin{equation}
K_{V^{\gamma}}=\frac{2}{\pi^{3}}\frac{e^{w\bar{\eta}}(1-z^{\prime}\bar{\zeta^{\prime}}+2e^{w\bar{\eta}} z\bar{\zeta})}{(1-e^{w\bar \eta}z\bar{\zeta}-z^{\prime}\bar{\zeta^{\prime}})^{4}}.
\end{equation}

In the next section, we will use (\ref{23}) and (\ref{24}) to obtain the boundary behavior of the Bergman kernel on the domains in Example 4.2 and 4.3. 
\section{Further Analysis of Examples 4.2 and 4.3}
In the following two sections, we focus on the boundary behavior of the Bergman kernel on $U^{\alpha}$ and $V^{\gamma}$. Our estimates for the kernel functions are on the diagonal. In this section, we use the explicit formulas of $K_{U^{\alpha}}$ and $K_{V^{\gamma}}$ from Example 4.2 and 4.3 and some admissible approach regions to analyze their boundary behavior. In the next section, we discuss more general cases without using explicit formulas. 

The boundary behavior of the Bergman kernel in the strongly pseudoconvex case is well understood. C. Fefferman\cite{Fefferman}
, L. Boutet de Monvel and J. Sj\" ostrand\cite{Monvel} gave an asymptotic expansion of the kernel function when the domain is bounded and strongly pseudoconvex. In the weakly  pseudoconvex case, the boundary behavior is difficult to analyze.  Near a weakly pseudoconvex point of finite type, certain estimates on the Bergman kernel were obtained by McNeal \cite{McNeal1,McNeal2}. Less is known near non-smooth boundary points. 

The simplest non-smooth case is the polydisk. 
Let $\Omega$ be the polydisk $\mathbb B^1\times\mathbb B^1$ in $\mathbb C^2$. Since the kernel function on a product domain is equal to the product of the kernel function on each factor, we have
$$K_{\Omega}(z_1,z_2;\bar{\zeta_1},\bar{\zeta_2})=\frac{1}{\pi^(1-z_1\bar\zeta_1)^2}\cdot\frac{1}{\pi(1-z_2\bar{\zeta_2})^2}.$$
If we approach the boundary point $p=(w_1,w_2)$ along the diagonal, then the boundary behavior of $K_{\Omega}$ depends on $w_1$ and $w_2$:
\begin{enumerate}
\item If $|w_1|=1$ and $|w_2|\neq1$, then in $\Omega$
$$\lim\limits_{z\rightarrow p}K_{\Omega}(z;\bar{z})(1-|z_1|^2)^2=\frac{1}{\pi^2(1-|w_2|^2)^2}\neq0.$$
\item If $|w_1|\neq1$ and $|w_2|=1$, then in $\Omega$
$$\lim\limits_{z\rightarrow p}K_{\Omega}(z;\bar{z})(1-|z_2|^2)^2=\frac{1}{\pi^2(1-|w_1|^2)^2}\neq0.$$
\item If $|w_1|=|w_2|=1$, then in $\Omega$
$$\lim\limits_{z\rightarrow p}K_{\Omega}(z;\bar{z})(1-|z_1|^2)^2(1-|z_2|^2)^2=\frac{1}{\pi^2}\neq0.$$
\end{enumerate}
In the 3rd case, $b\Omega$ is not smooth at the boundary points and the behavior of the Bergman kernel depends on the rate at which $|z_1|$ and $|z_2|$ tend to 1.
We will see similar phenomena when we analyze the boundary behavior of $K_{U^{\alpha}}$. 

\:\:
\paragraph{\bf {Example 4.2 revisited}}
The boundary of $U^{\alpha}$ is not smooth. $bU^{\alpha}$ at a point $(0,z^{\prime},w)$ where $\|z^{\prime}\|=|w|=1$. We let $\mathcal S_4$ denote the set of these non-smooth points. By calculating the Levi form of $U^{\alpha}$ on the smooth boundary points, one obtains that $(z,z^{\prime},w)$ is strongly pseudoconvex if both $\|z^{\prime}\|$ and $|w|$ are not equal to 1. We let $\mathcal S_1$ denote the set of these strongly pseudoconvex boundary points. We denote by $\mathcal S_2$ the set 
$$\{(0,z^{\prime},w)\in bU^{\alpha}:\|z^{\prime}\|=1,|w|\neq1\}$$
and denote by $\mathcal S_3$ the set
$$\{(0,z^{\prime},w)\in bU^{\alpha}:\|z^{\prime}\|\neq1,|w|=1\}.$$
Then $bU^{\alpha}=\mathcal S_1\cup \mathcal S_2\cup \mathcal S_3\cup \mathcal S_4$. The boundary behavior of the kernel function near the strongly pseudoconvex points in $\mathcal S_1$ is known. To obtain the result near the points in the other sets, we need an admissible approach region.
For $0<s<1$, let $\mathcal W_s$ denote the set
$$\{(z,z^{\prime},w)\in \mathbb C^n\times\mathbb C^m\times \mathbb C:|w|<1,\|z\|^{2s}+\|z^{\prime}\|^2+|w|^2<1+|w|^2\|z^{\prime}\|^2\}.$$
These sets exhaust $U^\alpha$ when $s$ tends to 1. Moreover, $\mathcal S_2$, $\mathcal S_3$ and $\mathcal S_4$ are contained in $b\mathcal W_s$. 
We will choose $\mathcal W_s$ as the admissible approach region. Let $r(z,z^\prime,w)$ denote the function:
$$1-\|z^{\prime}\|^2-\frac{\|z\|^{2}}{(1-|w|^2)}.$$
Then $U^{\alpha}$ can also be expressed as the set
$$\{(z,z^{\prime},w)\in\mathbb C^n\times\mathbb C^m\times \mathbb C:|w|<1,-r(z,z^{\prime},w)<0\}.$$
Note that the function $$\frac{\|z\|^{2s}}{(1-|w|^2)}$$ is bounded in $\mathcal W_s$. For $p=(0,z^{\prime}_0,w_0)\in \mathcal S_2\cup \mathcal S_3 \cup \mathcal S_4$, when approaching $p$ in $\mathcal W_s$, \begin{equation}\label{27}
\frac{\|z\|^{2}}{(1-|w|^2)}\rightarrow 0.
\end{equation}
Therefore $r(z,z^{\prime},w)$ is continuous in the closure of $\mathcal W_s$.
Combining (\ref{27}) and (\ref{23}), we obtain the following results on boundary behavior:
\begin{enumerate}
\item For $p_0=(0,z^{\prime}_0,w_0)\in \mathcal S_2$, the admissible limit
\begin{equation*}
	\lim\limits_{\mathcal W_s\ni p \rightarrow p_0}K_{U^{\alpha}}(p;\bar p)r^{n+m+1}(p)=\frac{(m+n)!(n+1)}{\pi^{m+n+1}(1-|w_0|^2)^{n+2}}\neq0
	.
\end{equation*}
\item For $p_0=(0,z^{\prime}_0,w_0)\in \mathcal S_3$, the admissible limit 
\begin{equation*}
\lim\limits_{\mathcal W_s\ni p \rightarrow p_0}K_{U^{\alpha}}(p;\bar p)(1-|w|^2)^{n+2}=\frac{(m+n)!(n+1)}{\pi^{m+n+1}r^{n+m+1}(p_0)}\neq0
.
\end{equation*}
\item For $p_0=(0,z^{\prime}_0,w_0)\in \mathcal S_4$, the admissible limit
\begin{equation*}
\lim\limits_{\mathcal W_s\ni p \rightarrow p_0}K_{U^{\alpha}}(p;\bar p)r^{n+m+1}(p)(1-|w|^2)^{n+2}=\frac{(m+n)!(n+1)}{\pi^{m+n+1}}\neq0
.
\end{equation*}
\end{enumerate}

\paragraph{\bf {Example 4.3 revisited}}
Calculating the Levi form shows that $V^{\gamma}$
 is a pseudoconvex domain. For any $w_0\in\mathbb C$ and $z^{\prime}_0\in \mathbb C^m$ on the unit sphere, $(0,z^{\prime}_0, w_0)$ is a weakly pseudoconvex point on $bV^{\gamma}$. With (\ref{24}), we can obtain the boundary behavior of the Bergman kernel function in an admissible approach region of $(0,z^{\prime}_0, w_0)$.

Let $0<s_j<1$ for $1\leq j\leq n$. Let $W_{\mathbf s}$ denote the domain 
\begin{equation}
\{(z,z^{\prime},w)\in \mathbb C^n\times\mathbb C^m\times\mathbb C:\sum_{j=1}^{n}e^{\gamma_j|w|^2}|z_j|^{2s_j}+\|z^{\prime}\|^2<1\}.
\end{equation}
For each $\mathbf s$, $W_{\mathbf s}$ is contained in $ V^{\gamma}$ and it exhausts $ V^{\gamma}$ as each $s_j$ approaches 1. Moreover, $bW_{\mathbf s}$ intersects $bV^{\gamma}$ at those weakly pseudoconvex points on $bV^{\gamma}$. 
Let $\rho$ denote the defining function of $V^{\gamma}$:
$$\rho(z,z^{\prime},w)=1-e^{|w|^2}\|z\|^2-\|z^{\prime}\|^2.$$
When approaching $p_0=(0,z^{\prime}_0, w_0)$ in the approach region $W_{\mathbf s}$, the admissible limit
\begin{equation}
\lim\limits_{
		W_{\mathbf s}\ni p \to p_0
		}
\frac{\sum_{j=1}^{n}e^{\gamma_j|w|^2}|z_j|^{2s_j}}{1-\|z^{\prime}\|^2}=0.
\end{equation}
Therefore,
\begin{equation}\label{30}
\lim\limits_{
		W_{\mathbf s}\ni p \to p_0
		}
\frac{\sum_{j=1}^{n}e^{\gamma_j|w|^2}|z_j|^{2s_j}}{\rho}=0.
\end{equation}
Applying (\ref{30}) to (\ref{24}), we have in $W_{\mathbf s}$:
\begin{equation}
\lim\limits_{
		W_{\mathbf s}\ni p \to p_0
		}{K_{V^{\gamma}}(p;\bar p)\rho^{m+n+1}}(p)=\frac{(m+n)!e^{n|w_0|^2}\sum_{j=1}^{n}\gamma_j}{\pi^{m+n+1}}\neq0.\nonumber
\end{equation}

\section{General Results for Boundary Behavior}
In section 5, we use the explicit formula of the Bergman kernel to study its boundary behavior at weakly pseudoconvex boundary points. In general, we do not require an explicit formula for the kernel function on the ``base" domain. If enough information on the boundary behavior of the kernel function of the ``base" domain is known, we can obtain the boundary behavior of the Bergman kernel on the ``target" domain. Here we'll discuss the boundary behavior for $U^{\alpha}$ and $V^{\gamma}$ when the ``base" domain $\Omega$ is smooth and strongly pseudoconvex. We hope in the future to extend the estimates in \cite{McNeal1,McNeal2} to this setting.

From now on, we let our ``base" domain $\Omega\subseteq\mathbb C^{n+m}$ be smooth, bounded, and $n$-star-shaped Hartogs with defining function $$r(|z_1|^2,\dots,|z_n|^2;z^{\prime},\bar{z}^{\prime})\in C^{\infty}(\bar \Omega).$$
Let $r_j$ denote the partial derivative of $r$ in the $j$'s component. We assume $r$ is non-decreasing in the first $n$ components, i.e. $r_j\geq0$ for $1\leq j\leq n$. Recall that $U^{\alpha}$ denotes the set
$$\Big\{(z,z^{\prime},w)\in\mathbb C^{n+m+1}:|w|<1,r\Big(\frac{|z_1|^2}{(1-|w|^2)^{\alpha_1}},\dots,\frac{|z_n|^2}{(1-|w|^2)^{\alpha_n}};z^{\prime},\bar{z}^{\prime}\Big)<0\Big\};$$
and $V^{\gamma}$ denotes the set
$$\Big\{(z,z^{\prime},w)\in\mathbb C^{n+m+1}:r\Big(e^{\gamma_1|w|^2}|z_1|^2,\dots,e^{\gamma_n|w|^2}|z_n|^2;z^{\prime},\bar{z}^{\prime}\Big)<0\Big\}.$$
To simplify the notation, we let $K_{\Omega}(z,z^{\prime})=K_{\Omega}(z,z^{\prime};\bar{z},\bar{z}^{\prime})$ and
let  \begin{align*}
&r_{U^{\alpha}}(z,z^{\prime},w)=r\Big(\frac{|z_1|^2}{(1-|w|^2)^{\alpha_1}},\dots,\frac{|z_n|^2}{(1-|w|^2)^{\alpha_n}};z^{\prime},\bar{z}^{\prime}\Big),\\
&r_{V^{\gamma}}(z,z^{\prime},w)=r\Big(e^{\gamma_1|w|^2}|z_1|^2,\dots,e^{\gamma_n|w|^2}|z_n|^2;z^{\prime},\bar{z}^{\prime}\Big).\end{align*}
We let $\nabla_z$ denote the partial gradient $(\frac{\partial}{\partial z_1},\dots,\frac{\partial}{\partial z_n})$.

We start with $U^{\alpha}$. The boundary behavior of the Bergman kernel on $U^{\alpha}$ is more complicated than on $V^{\alpha}$ for two reasons:
\begin{enumerate}
\item The possible non-smooth boundary points created by the two inequalities of $U^{\alpha}$. 
\item The singularity of $r_{U^{\alpha}}$ at points where $|w|=1$.
\end{enumerate}
The precise behavior at a boundary point depends on the geometry of $bU^{\alpha}$ there. We therefore stratify the boundary of $U^{\alpha}$ into four parts:
\begin{align*}
\mathcal S_1&=\{(z,z^{\prime},w)\in bU^{\alpha}:\nabla_z(r_{U^\alpha})\neq0\:\:{\rm and}\:\:|w|\neq1\},
\\\mathcal S_2&=\{(z,z^{\prime},w)\in bU^{\alpha}:\nabla_z(r_{U^\alpha})=0\:\:{\rm and}\:\:|w|\neq1\},
\\\mathcal S_3&=\{(z,z^{\prime},w)\in bU^{\alpha}:z=0,\:|w|=1\:\:{\rm and}\:\:(0,z^\prime)\notin b\Omega\},
\\\mathcal S_4&=\{(z,z^{\prime},w)\in bU^{\alpha}:z=0,\:|w|=1\:\:{\rm and}\:\:(0,z^\prime)\in b\Omega\}.
\end{align*}
By the boundedness of $\Omega$, we have
$$\{(z,z^{\prime},w)\in bU^{\alpha}:z\neq0,|w|=1\}=\emptyset.$$
Therefore $bU^{\alpha}=\mathcal S_1\cup \mathcal S_2\cup \mathcal S_3\cup \mathcal S_4$.
The points on $\mathcal S_1$ are strongly pseudoconvex and the boundary behaviors of the Bergman kernel near the boundary points of $\mathcal S_2$, $\mathcal S_3$ and $\mathcal S_4$ can be obtained by a suitable choice of approach regions. 

Let $x=(|z_1|^2,\dots,|z_n|^2)$. Since $\Omega$ is $n$-star-shaped Hartogs, we can let $L_{\Omega}$ denote the function such that
$$L_{\Omega}(x;z^{\prime},\bar{z}^{\prime})=K_{\Omega}(z,z^{\prime}).$$
We recall the result of C. Fefferman \cite{Fefferman} for bounded strongly pseudoconvex domain $\Omega$. There exist $ \Psi,\Phi\in C^{\infty}(\bar{\Omega})$, such that
\begin{equation}\label{31}
L_{\Omega}(x;z^{\prime},\bar{z}^{\prime})=\frac{\Psi(x;z^{\prime},\bar{z}^{\prime})}{(-r)^{n+m+1}(x;z^{\prime},\bar{z}^{\prime})}+\Phi(x;z^{\prime},\bar{z}^{\prime})\log (-r(x;z^{\prime},\bar{z}^{\prime})).
\end{equation}
 Applying (\ref{31}) and Theorem 1, we obtain the following result on the ``target" domain $U^{\alpha}$.
\begin{thm}
	Let $\Omega$ and $U^{\alpha}$ be as above. Suppose $\Omega$ is strongly pseudoconvex. Then $U^{\alpha}$ is pseudoconvex. The point $p\in bU^{\alpha}$ is a strongly pseudoconvex point if $p\in \mathcal S_1$. Near the points of $\mathcal S_2$, $\mathcal S_3$ and $\mathcal S_4$, the kernel function behaves in three different ways: \begin{enumerate}
\item For $(z_0,z^{\prime}_0,w_0)\in \mathcal S_2$, there exists an admissible approach region $W_2$ of $(z_0,z^{\prime}_0,w_0)$ such that when approaching $(z_0,z^{\prime}_0,w_0)$ in $W_2$,
\begin{equation}\label{32}
K_{U^{\alpha}}(z,z^{\prime},w)(-r_{U^{\alpha}})^{m+n+1}(z,z^{\prime},w)
\end{equation}
has a nonzero limit.
\item For $(z_0,z^{\prime}_0,w_0)\in \mathcal S_3$, there exists an admissible approach region $W_3$ of $(z_0,z^{\prime}_0,w_0)$ such that when approaching $(z_0,z^{\prime}_0,w_0)$ in $W_3$,
\begin{equation}\label{33}
K_{U^{\alpha}}(z,z^{\prime},w)(1-|w|^2)^{2+\alpha\cdot\mathbf 1}
\end{equation}
has a nonzero limit.
\item For $(z_0,z^{\prime}_0,w_0)\in \mathcal S_4$, there exists an admissible approach region $W_4$ of $(z_0,z^{\prime}_0,w_0)$ such that when approaching $(z_0,z^{\prime}_0,w_0)$ in $W_4$,
\begin{equation}\label{34}
K_{U^{\alpha}}(z,z^{\prime},w)(1-|w|^2)^{2+\alpha\cdot\mathbf 1}(-r_{U^{\alpha}})^{m+n+1}(z,z^{\prime},w)
\end{equation}
has a nonzero limit.
	\end{enumerate}	
\end{thm}
\begin{proof} 
Let $X=(X_1,\dots,X_n)$ denote the vector $$\Big(\frac{|z_1|^2}{(1-|w|^2)^{\alpha_1}},\dots,\frac{|z_n|^2}{(1-|w|^2)^{\alpha_n}}\Big).$$
Since the range of $X$ on $U^{\alpha}$ is the same as the range of $x$ on $\Omega$, we can replace $x$ in (\ref{31}) by $X$ and have 
\begin{equation}\label{35}
L_{\Omega}(X;z^{\prime},\bar{z}^{\prime})=\frac{\Psi(X;z^{\prime},\bar{z}^{\prime})}{(-r)^{n+m+1}(X;z^{\prime},\bar{z}^{\prime})}+\Phi(X;z^{\prime},\bar{z}^{\prime})\log (-r(X;z^{\prime},\bar{z}^{\prime}))
\end{equation}
with $\Psi(X;z^{\prime},\bar{z}^{\prime}),\Phi(X;z^{\prime},\bar{z}^{\prime})\in C^{\infty}(U^{\alpha})$. Using change of variables formula yields 
$$L_{\Omega}(X;z^{\prime},\bar{z}^{\prime})=(1-|w|^2)^{\alpha\cdot\mathbf 1}K_{U^{\alpha}_{w}}(z,z^{\prime}).$$
Then Theorem 1 implies that
\begin{equation}\label{36}
K_{U^{\alpha}}(z,z^{\prime},w)=(c_{\alpha}I+D)\frac{L_{\Omega}(X;z^{\prime},\bar{z}^{\prime})}{\pi(1-|w|^2)^{2+\alpha\cdot\mathbf 1}},
\end{equation}
where $c_{\alpha}=(1+\sum_{j=1}^{n}\alpha_j)$ and  $D=\sum_{j=1}^{n}\alpha_jz_j\frac{\partial}{\partial z_j}$. 

Note that $r(X;z^{\prime},\bar{z}^{\prime})$ is equal to $r_{U^{\alpha}}(z,z^{\prime},w)$. Multiplying both sides of (\ref{36}) by $(1-|w|^2)^{2+\alpha\cdot\mathbf 1}(-r_{U^{\alpha}})^{m+n+1}(z,z^{\prime},w)$, (\ref{34}) becomes
\begin{equation}\label{j}
\pi^{-1}(-r)^{n+m+1}(X;z^{\prime},\bar{z}^{\prime})(c_{\alpha}I+D)L_{\Omega}(X;z^{\prime},\bar{z}^{\prime}).
\end{equation}
We set $I_1+I_2$ equals (\ref{j}) where 
\begin{align*}
 &I_1=\pi^{-1}(-r)^{n+m+1}(X;z^{\prime},\bar{z}^{\prime})c_{\alpha}L_{\Omega}(X;z^{\prime},\bar{z}^{\prime}),\nonumber\\&I_2=\pi^{-1}(-r)^{n+m+1}(X;z^{\prime},\bar{z}^{\prime})D\big(L_{\Omega}(X;z^{\prime},\bar{z}^{\prime})\big)\nonumber.\end{align*}
Applying (\ref{35}) to $I_1$, we have
\begin{equation}\label{38}\pi I_1(X;z^{\prime},\bar{z}^{\prime})=c_{\alpha}\big(\Psi+(-r)^{n+m+1}\Phi\log(-r)\big).
\end{equation}
Applying the product rule to $I_2$, we have
\begin{align}
\pi I_2(X;z^{\prime},\bar{z}^{\prime})\nonumber=&D\big((-r)^{n+m+1}L_{\Omega}\big)-L_{\Omega}D(-r)^{n+m+1}\nonumber.
\end{align}
We set $J_1=D\big((-r)^{n+m+1}L_{\Omega}\big)$ and $J_2=L_{\Omega}D(-r)^{n+m+1}$, then $\pi I_2=J_1-J_2$. Substituting  (\ref{35}) to $J_1$ and $J_2$ yields
\begin{align}\label{40}
J_1(X;z^{\prime},\bar{z}^{\prime})=&D\Psi+(-r)^{n+m+1}\log(-r)Dt\nonumber\\&+\big(1+(n+m+1)\log(-r)\big)(-r)^{n+m}D(-r).
\end{align}
and 
\begin{align}\label{41}
J_2(X;z^{\prime},\bar{z}^{\prime})&=(n+m+1)\big((-r)^{n+m+1}L_{\Omega}\big)\frac{D(-r)}{-r}\nonumber\\&=(n+m+1)\big(\Psi+\Phi(-r)^{n+m+1}\log(-r)\big)\frac{D(-r)}{-r}.
\end{align}
Let $p=(z_0,z_0^{\prime},w_0)$ be a boundary point $U^{\alpha}$. When $|w_0|\neq 1$, we let  $X_0$ denote the corresponding vector $X$ at point $p$.

\paragraph{\bf Case 1)}For $(z_0,z^{\prime}_0,w_0)\in \mathcal S_2$, we have $\nabla_zr_{U^{\alpha}}(p)=0$, $(X_0,z_0^{\prime})\in \partial\Omega$, and $|w_0|\neq1$. Then a nonzero limit of (\ref{34}) exists is equivalent to a nonzero limit of (\ref{32}) exists. Since $|w_0|\neq1$, $X$ is smooth near $p$. Thus $r(X;z^{\prime},\bar{z}^{\prime})$ is smooth in a neighborhood of $p$ and has limit $r(X_0;z_0^{\prime},\bar{z_0^{\prime}})=0$. Since $(-r)\log(-r)$ also has limit equals zero at point $p$, the limit of $I_1$ and $J_1$ exist. To achieve the limit existence of $J_2$ at $p$, we need an admissible approach region in which the limit of $\frac{D(-r)}{-r}$ equals zero. Let $r_j(X;z^{\prime},\bar{z}^{\prime})$ be the partial derivative of $r$ in the $j$th component. For $0<q<1$, we consider the following approach region
$$W_2=\Big\{(z,z^{\prime},w)\in U^{\alpha}:\sum_{j=1}^{n}\big(|z_j|^2r_j(X;z^{\prime},\bar{z}^{\prime})\big)^q<-r(X;z^{\prime},\bar{z}^{\prime})\Big\}.$$ 
We show $W_2$ is not empty. Since $\nabla_zr_{U^{\alpha}}(p)=0$, we have
$$\frac{\partial}{\partial z_j}r(X;z^{\prime},\bar{z}^{\prime})=\frac{\bar z_j}{(1-|w|^2)^{\alpha_j}}r_j(X;z^{\prime},\bar{z}^{\prime})=0,$$
for all $j$ at $p$. Thus $\sum_{j=1}^{n}\big(|z_j|^2r_j(X;z^{\prime},\bar{z}^{\prime})\big)^q=0$ when approaching $p$ along the normal direction of $bU^{\alpha}$. This observation implies that $W_2$ is not empty and $p\in bW_2$.  By perhaps shrink $W_2$, we may consider $W_2$ as a connected set. Note that
\begin{align}
\Big|\frac{D(-r)(X;z^{\prime},\bar{z}^{\prime})}{-r(X;z^{\prime},\bar{z}^{\prime})}\Big|&=\frac{-\sum_{j=1}^{n}\alpha_j\frac{|z_j|^2}{(1-|w|^2)^{\alpha}}r_j(X;z^{\prime},\bar{z}^{\prime})}{-r(X;z^{\prime},\bar{z}^{\prime})}\nonumber\\&<\frac{c\sum_{j=1}^{n}|z_j|^{2}r_j(X;z^{\prime},\bar{z}^{\prime})}{-r(X;z^{\prime},\bar{z}^{\prime})}
\end{align}
for some constant $c>0$. In $W_2$,
$$\frac{\sum_{j=1}^{n}\big(|z_j|r_j(X;z^{\prime},\bar{z}^{\prime})\big)^q}{-r(X;z^{\prime},\bar{z}^{\prime})}<1.$$
When approaching boundary point $p$ inside $W_2$, we have 
\begin{align}
&\frac{\sum_{j=1}^{n}|z_j|^{2}r_j(X;z^{\prime},\bar{z}^{\prime})}{-r(X;z^{\prime},\bar{z}^{\prime})}\nonumber\\\leq&\frac{\sum_{k=1}^{n}\big(|z_j|^2r_j(X;z^{\prime},\bar{z}^{\prime})\big)^{1-q}\sum_{j=1}^{n}\big(|z_j|^2r_j(X;z^{\prime},\bar{z}^{\prime})\big)^q}{-r(X;z^{\prime},\bar{z}^{\prime})}\nonumber\\<&\sum_{k=1}^{n}\big(|z_j|^2r_j(X;z^{\prime},\bar{z}^{\prime})\big)^{1-q}\rightarrow 0.
\end{align}
Hence $J_1$ and $J_2$ in (\ref{40}) and (\ref{41}) have admissible limit zero at point $p$. By the strong pseudoconvexity of $\Omega$, (\ref{38}) has nonzero limit. Therefore in $W_2$, the limit of (\ref{32}) at point $p$ exists and is not equal to zero.
\paragraph{\bf Case 2)}For $(z_0,z^{\prime}_0,w_0)\in \mathcal S_3$, we have $z_0=0$, $(0,z_0^{\prime})\notin \partial\Omega$, and $|w_0|=1$. We consider the region
$$W_3=\Big\{(z,z^{\prime},w)\in U^{\alpha}:\frac{|z_j|^2}{(1-|w|^2)^{p_j}}<1,\forall \;1\leq j\leq n\Big\}$$
where $p_j>\alpha_j$ for all $j$. Similar reasoning as above implies that $W_3$ is nonempty and connected. When we approaching the bounadary point $p$ in $W_3$,
$$\frac{|z_j|^2}{(1-|w|^2)^{\alpha_j}}=\frac{|z_j|^2(1-|w|^2)^{p_j-\alpha_j}}{(1-|w|^2)^{p_j}}<(1-|w|^2)^{p_j-\alpha_j}\rightarrow 0.$$
Thus $X$, $D\Psi(X;z^{\prime},\bar{z}^{\prime})$, $D\Phi(X;z^{\prime},\bar{z}^{\prime})$ and $D(-r(X;z^{\prime},\bar{z}^{\prime}))$ all tends to zero at $p$. Since $(0,z_0^{\prime})\notin \partial\Omega$, the function $-r(X,z^{\prime},\bar{z}^{\prime})$ has a positive limit at point $p$. Plugging these results into (\ref{40}) and (\ref{41}), we have both $J_1$ and $J_2$ tend to zero.
The limit of (\ref{38}) is positive since $I_1=c_\alpha L_{\Omega}$ and $L_{\Omega}$ is positive at $(0,z^{\prime}_0,\bar{z^{\prime}_0})$.
Therefore when approaching $p$ in $W_3$, (\ref{34}) and $r_{U^{\alpha}}$ has a nonzero limit. Hence the limit of (\ref{33}) is also not zero.
\paragraph{\bf Case 3)} When $(z_0,z^{\prime}_0,w_0)\in \mathcal S_4$, we have $z_0=0$, $(0,z_0^{\prime})\in \partial\Omega$, and $|w_0|=1$. Consider the approach region $W_4=W_2\bigcap W_3$.
Since both $W_2$ and $W_3$ contains the set $Z\{z_1,\dots,z_n\}\bigcap U^{\alpha}$ and $p\in Z\{z_1,\dots,z_n\}\bigcap \overline {U^{\alpha}}$, we can approach $p$ inside $W_4$.
By our previous results, when we tend to $p$ in $W_4$,
 $X$, $D\Psi(X;z^{\prime},\bar{z}^{\prime})$, $D\Phi(X;z^{\prime},\bar{z}^{\prime})$, $D(-r(X;z^{\prime},\bar{z}^{\prime}))$, $r(X,z^{\prime},\bar{z}^{\prime})$, $r(X,z^{\prime},\bar{z}^{\prime})\log(-r(X,z^{\prime},\bar{z}^{\prime}))$ and $\frac{D(-r(X;z^{\prime},\bar{z}^{\prime}))}{-r(X;z^{\prime},\bar{z}^{\prime})}$ all tends to zero.
Hence the limit of $J_1$ and $J_2$ equals zero and the limit of $I_1$ is equal to a nonzero constant, and (\ref{34}) has a nonzero admissible limit in $W_4$.

\paragraph{\bf Case 4)}For the boundary points in $\mathcal S_1$, the strong pseudoconvexity can be obtained by calculating the Levi form.\qedhere
\end{proof}

Compared to $U^{\alpha}$, the boundary behavior of the kernel funcion $V^{\gamma}$ is simpler. The argument is similar to the proof of Theorem 3. We state the result without proof.
\begin{thm}
Let $\Omega$ and $V^{\gamma}$ be as above. Suppose $\Omega$ is bounded and strongly pseudoconvex. Then $V^{\gamma}$ is pseudoconvex. The point $p=(z_0,z^{\prime}_0,w_0)\in bV^{\gamma}$ is a weakly pseudoconvex point if $\nabla_z(r_{V^{\gamma}})(p)=0$. Moreover, for weakly pseudoconvex point $p$, we can find an admissible approach region $W$, such that when approaching $p$ inside $W$:
$$K_{V^{\gamma}}(z,z^{\prime},w)(-r_{V^{\gamma}})^{n+m+1}(z,z^{\prime},w)$$tends to a nonzero constant.
\end{thm}
\begin{re*}In both Theorems 3 and 4 above, we assumed the existence of $z^{\prime}$ components. Because of our assumption, The points in $\mathcal S_2$, $\mathcal S_3$ of Theorem 3 and the weakly pseudoconvex boundary points in $bV^\gamma$ of Theorem 4 are of infinite type in the sense of D'Angelo. If there is no $z^{\prime}$, i.e. $m=0$ in the definition of $\Omega$, then the boundary geometry of the target domains is different. In this case, $V^{\gamma}$ becomes a strongly pseudoconvex domain. The boundary geometry of $U^{\alpha}$, on the other hand, depends on the value of $\alpha$. One can see this immediately from Example 4.1.\end{re*}
\section{Higher Dimensional Analogues}
In Theorems 1 and 2, we apply a first order differential operator to the Bergman kernel function on the ``base" domain to obtain the kernel function on certain domains in one higher dimension. By Lemma 3.1, the ``target" domains $U^{\alpha}$ and $V^{\gamma}$ are also $(n+1)$-star-shaped Hartogs. Therefore we can repeat using Theorems 1 and 2 to obtain the Bergman kernel on more complicated domains.
\begin{ex}[Repeated use of Theorems 1 and 2]
	The diagram below indicates how to obtain the kernel function explicitly on increasingly complicated domains.
\begin{gather}
\{z\in \mathbb C:|z|^2<1\}\nonumber
\\\Downarrow\nonumber
\\\{z\in \mathbb C^2:|z_1|^{2p}+|z_2|^2<1\}\nonumber
\\\Downarrow\nonumber
\\\{z\in \mathbb C^3:|z_1|^{2p}+\exp\{|z_3|^2\}|z_2|^2<1\}\nonumber
\\\Downarrow\nonumber
\\\Big{\{}z\in \mathbb C^4:|z_1|^{2p_1}+\exp\{\frac{|z_3|^2}{(1-|z_4|)^{p_2}}\}|z_2|^2<1,|z_4|<1\Big\}\nonumber
\\\Downarrow\nonumber
\\\Big\{z\in \mathbb C^5:\frac{|z_1|^{2p_1}}{(1-|z_5|^2)^{p_3}}+\exp\{\frac{|z_3|^2}{(1-|z_4|)^{p_2}}\}|z_2|^2<1,|z_4|<1,|z_5|<1\Big\}\nonumber
\\\Downarrow\nonumber
\\\Big\{z\in \mathbb C^6:\frac{|z_1|^{2p_1}}{(1-e^{|z_6|^2}|z_5|^2)^{p_3}}+\exp\{\frac{|z_3|^2}{(1-|z_4|)^{p_2}}\}|z_2|^2<1,|z_4|<1,e^{|z_6|^2}|z_5|^2<1\Big\}\nonumber
\\\Downarrow\nonumber
\\\vdots\nonumber
\end{gather}
The Bergman kernels in the first two cases are known. The kernel in the third case is equal to
\begin{align*}
&\frac{e^{z_3\bar\zeta_3}}{\pi^3p}\bigg(\frac{(1+p)(1-e^{z_3\bar{\zeta_3}}z_2\bar{\zeta_2})^{\frac{1}{p}}+(1-p)z_1\bar{\zeta_1}}{(1-e^{z_3\bar{\zeta_3}}z_2\bar{\zeta_2})^{2-\frac{1}{p}}((1-e^{z_3\bar{\zeta_3}}z_2\bar{\zeta_2})^{\frac{1}{p}}-z_1\bar{\zeta_1})^{3}}\\&+\frac{(p-1)e^{z_3\bar{\zeta_3}}z_2\bar{\zeta_2}\big((2+\frac{1}{p})(1-e^{z_3\bar{\zeta_3}}z_2\bar{\zeta_2})^{\frac{1}{p}}+(2-\frac{1}{p})z_1\bar{\zeta_1}\big)}{(1-e^{z_3\bar{\zeta_3}}z_2\bar{\zeta_2})^{3-\frac{1}{p}}((1-e^{z_3\bar{\zeta_3}}z_2\bar{\zeta_2})^{\frac{1}{p}}-z_1\bar{\zeta_1})^{3}}\\&+\frac{2e^{z_3\bar{\zeta_3}}z_2\bar{\zeta_2}\big((2+\frac{1}{p})(1-e^{z_3\bar{\zeta_3}}z_2\bar{\zeta_2})^{\frac{1}{p}}+(2-\frac{2}{p})z_1\bar{\zeta_1}\big)}{(1-e^{z_3\bar{\zeta_3}}z_2\bar{\zeta_2})^{3-\frac{2}{p}}((1-e^{z_3\bar{\zeta_3}}z_2\bar{\zeta_2})^{\frac{1}{p}}-z_1\bar{\zeta_1})^{4}}\bigg).
\end{align*}
For the domains below the third case, the kernel functions are more complicated and we will omit them here.
\end{ex}
We can generalize Theorems 1 and 2 when our ``target" domains involves $w\in \mathbb C^k$ instead of a single variable.
Let $\Omega\subseteq \mathbb C^{n+m}$ be $n$-star-shaped Hartogs in the first $n$ variables.
Consider the ``target" domains:
\begin{itemize}
\item
$U^{\alpha}=\{(z,z^{\prime},w)\in \mathbb C^{n+m}\times \mathbb C^k:(f_{\alpha}(z,w),z^{\prime})\in \Omega,\|w\|^2<1\}$

where  $$f_{\alpha}(z,w)=\Big(\frac{z_1}{(1-\|w\|^2)^{\frac{\alpha_1}{2}}},\dots,\frac{z_n}{(1-\|w\|^2)^{\frac{\alpha_n}{2}}}\Big)$$
and $\alpha_j$'s are positive numbers.

\item
$V^{\gamma}=\{(z,z^{\prime},w)\in \mathbb C^{n+m}\times \mathbb C^k:(g_{\gamma}(z,w),z^{\prime})\in \Omega\}$

where
$$g_{\gamma}(z,w)=\Big(e^{\frac{\gamma_1\|w\|^{2}}{2}}z_1,\dots,e^{\frac{\gamma_n\|w\|^{2}}{2}}z_n\Big)$$
and $\gamma_j$'s are positive numbers.
\end{itemize}	
Since we can construct a diagram from $\Omega$ to $V^{\gamma}$ in Example 7.1, the kernel function $K_{V^{\gamma}}$ can be obtained directly by repeatly applying Theorem 2:
\begin{thm}
	For $(z,z^{\prime},w;\zeta,\zeta^{\prime},\eta) \in V^{\gamma}\times V^{\gamma}$,
	\begin{equation}
	K_{V^{\gamma}}(z,z^{\prime},w;\bar \zeta,\bar \zeta^{\prime},\bar \eta)=D_{V^{\gamma}}K_{V^{\gamma}_{\eta}}(l(z,w,\eta),z^{\prime};\bar\zeta,\bar{\zeta^{\prime}})
	\end{equation}
	where 
	$$l(z,w,\eta)=\Big(z_1e^{\gamma_1(\langle w,{\eta}\rangle-\|\eta\|^2)},\dots,z_ne^{\gamma_n(\langle w{\eta}\rangle-\|\eta\|^2)}\Big)$$
	and $D_{V^{\gamma}}$ is the $k$-th order differential operator defined by $$D_{V^{\gamma}}=\frac{e^{(\gamma\cdot\mathbf 1)(\langle w,{\eta}\rangle-\|\eta\|^2)}}{\pi^k}\Big(\sum_{j=1}^{n}\gamma_j(I +z_j\frac{\partial}{\partial z_j})\Big)^k.$$
\end{thm}
 The trick used in Example 7.1 does not work for $U^{\alpha}$. Nevertheless, we have the following result:
\begin{thm}
For $(z,z^{\prime},w;\zeta,\zeta^{\prime},\eta)\in U^{\alpha}\times U^{\alpha}$,
\begin{equation}\label{45}
K_{U^{\alpha}}(z,z^{\prime},w;\bar \zeta,\bar \zeta^{\prime},\bar \eta)=D_{U^{\alpha}}K_{U^{\alpha}_{\eta}}(h(z,w,\eta),z^{\prime};\bar\zeta,\bar{\zeta^{\prime}})
\end{equation}
where 
$$h(z,w,\eta)=\Big(z_1(\frac{1-\|\eta\|^2}{1-\langle w, \eta\rangle})^{\alpha_1},\dots,z_n(\frac{1-\|\eta\|^2}{1-\langle w, \eta\rangle})^{\alpha_n}\Big)$$
and $D_{U^{\alpha}}$ is the $k$-th order differential operator defined by $$D_{U^{\alpha}}=\frac{(1-\|\eta\|^2)^{\alpha\cdot\mathbf 1}}{\pi^k(1-\langle w, \eta\rangle)^{k+1+\alpha\cdot\mathbf 1}}\prod_{j=1}^{k}\Big(jI+\sum_{l=1}^{n}\alpha_{l}(I+z_{l}\frac{\partial}{\partial z_{l}})\Big).$$
\end{thm}

\begin{proof}
Let $K_1(z,z^{\prime},w;\bar \zeta,\bar \zeta^{\prime},\bar \eta)$ denote the right-hand side of (\ref{45}).
By the same argument in the proof of Theorem 1, $K_1(z,z^{\prime},w;\bar \zeta,\bar \zeta^{\prime},\bar \eta)$ is defined on $U^{\alpha}\times U^{\alpha}$ and has the expansion:
\begin{equation*}
\sum_{\mathbf a,\mathbf b,\mathbf c}{c_{\mathbf a,\mathbf b,\mathbf c}(z\bar \zeta)^{\mathbf a}\phi_{\mathbf a,\mathbf b}(z^{\prime})\overline{\phi_{\mathbf a,\mathbf b}(\zeta^{\prime})}(w\bar{\eta})^{\mathbf c}}.
\end{equation*}
For arbitrary $z^\mathbf a\phi_{\mathbf a,\mathbf b}(z^\prime)w^{\mathbf c}\in A^2(U^{\alpha})$,
\begin{align}\label{46}
&\int_{U^{\alpha}}K_1(z,z^{\prime},w;\bar{\zeta},\bar{\zeta^{\prime}},\bar{\eta})\zeta^\mathbf a\phi_{\mathbf a,\mathbf b}(\zeta^\prime)\eta^{\mathbf c}dV\nonumber
\\=&\int_{\mathbb B^k}\eta^{\mathbf c}\int_{U^{\alpha}_\eta}D_{U^\alpha}K_{U^{\alpha}_{\eta}}\big(h(z,w,\eta),z^{\prime};\bar{\zeta};\bar{\zeta^{\prime}}\big)\zeta^\mathbf a\phi_{\mathbf a,\mathbf b}(\zeta^\prime)dV(\zeta,\zeta^{\prime})dV(\eta).\;\;\;\;\;\;\;
\end{align}
Using the reproducing property of $K_{U^{\alpha}_\eta}$ on $U^{\alpha}_\eta$, we have
\begin{align}
&\int_{U^{\alpha}_\eta}D_{U^\alpha}K_{U^{\alpha}_{\eta}}(h(z,w,\eta),z^{\prime};\bar{\zeta};\bar{\zeta^{\prime}})\zeta^\mathbf a\phi_{\mathbf a,\mathbf b}(\zeta^\prime)dV(\zeta,\zeta^{\prime})\nonumber
\\=&\Big(\prod_{{j}=1}^{k}\big({j}+\alpha\cdot(\mathbf a +\mathbf 1)\big)\Big)\frac{(1-\|\eta\|^2)^{\alpha\cdot\mathbf 1}}{\pi(1-\langle w,\eta\rangle)^{2+\alpha\cdot\mathbf 1}}h(z,w,\eta)^{\mathbf a}\phi_{\mathbf a,
	\mathbf b}(z^{\prime}).
\end{align}
Therefore the integral in the last line  of (\ref{46}) becomes
\begin{equation}\label{48}
\Big(\prod_{{j}=1}^{k}\big({j}+\alpha\cdot(\mathbf a +\mathbf 1)\big)\Big)\phi_{\mathbf a,\mathbf b}(z^\prime)\int_{\mathbb B^1}\frac{(1-\|\eta\|^2)^{\alpha\cdot\mathbf 1}\eta^{\mathbf c}h(z,w,\eta)^{\mathbf a}}{\pi^k(1-\langle w,\eta\rangle)^{1+k+\alpha\cdot \mathbf 1}}dV(\eta).
\end{equation}
Since
$h(z,w,\eta)=\Big(z_1(\frac{1-\|\eta\|^2}{1-\langle w,\eta\rangle})^{\alpha_1},\dots,z_n(\frac{1-\|\eta\|^2}{1-\langle w,\eta\rangle})^{\alpha_n}\Big)$, (\ref{48}) equals
\begin{equation}\label{49}
\frac{\prod_{{j}=1}^{k}\big({j}+\alpha\cdot(\mathbf a +\mathbf 1)\big)z^{\mathbf a}\phi_{\mathbf a,\mathbf b}(z^\prime)}{\pi^k}\int_{\mathbb B^k}\frac{(1-\|\eta\|^2)^{\alpha\cdot(\mathbf a+\mathbf 1)}\eta^{\mathbf c}}{(1-\langle w,\eta\rangle)^{1+k+\alpha\cdot (\mathbf a+\mathbf 1)}}dV(\eta).
\end{equation}
Expanding the denominator in (\ref{49}), we have
\begin{align}\label{50}
(\ref{49})=&z^{\mathbf a}\phi_{\mathbf a,\mathbf b}(z^\prime)\int_{\mathbb B^k}\sum_{\mathbf p}\frac{\big(1+\alpha\cdot (\mathbf a+\mathbf 1)\big)_{(\mathbf j\cdot \mathbf 1)+k}(1-\|\eta\|^2)^{\alpha\cdot(\mathbf a+\mathbf 1)}(w\bar{\eta})^{\mathbf p}}{\pi^k \prod_{p=1}^{k}(p_j)!}\eta^{\mathbf c}dV\nonumber
\\=& z^{\mathbf a}\phi_{\mathbf a,\mathbf b}(z^\prime)w^{\mathbf c}\int_{\mathbb B^k}\frac{\big(1+\alpha\cdot (\mathbf a+\mathbf 1)\big)_{({\mathbf c\cdot \mathbf 1})+k}(1-\|\eta\|^2)^{\alpha\cdot(\mathbf a+\mathbf 1)}{\eta}^{\mathbf c}\bar{\eta}^{\mathbf c}}{\pi^k \prod_{j=1}^{k}(c_j)!}dV.
\end{align}
By letting $r_j=|\eta_j|^2$, we have
\begin{equation}\label{51}
\int_{\mathbb B^k}(1-\|\eta\|^2)^{\alpha\cdot(\mathbf a+\mathbf 1)}|{\eta}|^{2\mathbf c}dV
=\pi^k\int_{\mathbf B^k_{+}}(1-\sum_{j=1}^{k}r_j)^{\alpha\cdot(\mathbf a+\mathbf 1)}r^{\mathbf c}dV,
\end{equation}
where $\mathbf B^k_{+}=\{(r_1,\dots,r_k)\in \mathbb R^k_+:\sum_{j=1}^{k}r_j<1\}$.
We claim
\begin{equation}\label{52}
\pi^k\int_{\mathbf B^k_{+}}(1-\sum_{j=1}^{k}r_j)^{\alpha\cdot(\mathbf a+\mathbf 1)}r^{\mathbf c}dV=\frac{\pi^k\prod_{j=1}^{k}(c_j)!}{\big(1+\alpha\cdot(\mathbf a +\mathbf 1)\big)_{(\mathbf c\cdot\mathbf 1)+k}}.
\end{equation}
Assuming the claim, then (\ref{49}) equals
$z^{\mathbf a}\phi_{\mathbf a,\mathbf b}(z^\prime)w^{\mathbf c}$
which completes the proof.

To prove (\ref{52}), we do induction on $k$. When $k=1$, we have
\begin{equation}
\int_{0}^{1}(1-r)^{\alpha\cdot(\mathbf a+\mathbf 1)}r^{c}dV=\frac{\Gamma\big(1+\alpha\cdot(\mathbf a+\mathbf 1)\big)\Gamma(c+1)}{\Gamma\big(2+c+\alpha\cdot(\mathbf a+\mathbf 1)\big)},\nonumber
\end{equation}
and (\ref{52}) holds.
Suppose (\ref{52}) holds when $k<N$. For $k=N$,
\begin{align}\label{53}
&\int_{\mathbb B^N_{+}}(1-\sum_{j=1}^{N}r_j)^{\alpha\cdot(\mathbf a+\mathbf 1)}r^{\mathbf c}dV\nonumber
\\=&\int_{0}^{1}{r_N}^{c_N}\int_{W_{r_N}}(1-\sum_{j=1}^{N}r_j)^{\alpha\cdot(\mathbf a+\mathbf 1)}\prod_{j=1}^{N-1}r_j^{c_j}dr_1\dots dr_{N-1}dr_N,
\end{align}
where $W_{r_N}=\{(r_1,\dots,r_{N-1})\in \mathbb R^{N-1}_+:\sum_{j=1}^{N-1}r_j<1-r_N\}$.
By substituting $t_j=\frac{r_j}{1-r_N}$  for $1\leq j\leq N-1$ to the integral in the second line of (\ref{53}), we obtain
\begin{align}\label{54}
&\Big(\int_{0}^{1}{r_N}^{c_N}(1-r_N)^{\alpha\cdot(\mathbf a+\mathbf 1)+\sum_{j=1}^{N-1}(c_j+1)}dr_N\Big)\nonumber
\\&\times\Big(\int_{\mathbb B^{N-1}_+}(1-\sum_{j=1}^{N-1}t_j)^{\alpha\cdot(\mathbf a+\mathbf 1)}\prod_{j=1}^{N-1}t_j^{c_j}dt_1\dots dt_{N-1}\Big).
\end{align}
Using the definition of the beta function and the induction hypothesis yields
\begin{align}
(\ref{54})=&\frac{\Gamma(c_N+1)\Gamma\big(\alpha\cdot(\mathbf a+\mathbf 1)+\sum_{j=1}^{N-1}(c_j+1)+1\big)}{\Gamma\big(\alpha\cdot(\mathbf a+\mathbf 1)+\sum_{j=1}^{N}(c_j+1)+1\big)}\nonumber
\\&\times\frac{\prod_{j=1}^{N-1}(c_j)!}{\big(1+\alpha\cdot(\mathbf a +\mathbf 1)\big)_{\sum_{j=1}^{N-1}(c_j+1)}}\nonumber
\\=&\frac{\prod_{j=1}^{N}(c_j)!}{(\alpha\cdot\big(\mathbf a+\mathbf 1)+\sum_{j=1}^{N-1}c_j+N\big)_{c_N+1}\big(1+\alpha\cdot(\mathbf a +\mathbf 1)\big)_{\sum_{j=1}^{N-1}(c_j+1)}}\nonumber
\\=&\frac{\prod_{j=1}^{N}(c_j)!}{\big(1+\alpha\cdot(\mathbf a +\mathbf 1)\big)_{\mathbf c\cdot \mathbf 1+N}}\nonumber.
\end{align}
Therefore (\ref{52}) holds for all $k$.
\end{proof}	
\bibliographystyle{plain}
\bibliography{manuscript} 
\end{document}